\documentclass{article}

\usepackage[english,USenglish,UKenglish]{babel}
\usepackage{amssymb}
\usepackage{amstext}
\usepackage{amsmath}

\usepackage{amsmath}
\usepackage{txfonts}
\usepackage{theorem}
\usepackage{graphicx}
\usepackage[all]{xy}
\usepackage{latexsym}

\newtheorem{defn}{Definition}[section]
\newtheorem{thm}[defn]{Theorem}
\newtheorem{prop}[defn]{Proposition}
\newtheorem{lem}[defn]{Lemma}

\newtheorem{cor}[defn]{Corollary}
\newtheorem{prob}[defn]{Problem}
{\theorembodyfont{\upshape} \newtheorem{rem}[defn]{Remark}}
{\theorembodyfont{\upshape} }
{\theorembodyfont{\upshape} \newtheorem{proof}{Proof}}

\begin{document}

\title{On the Fukaya-Seidel categories of surface Lefschetz fibrations}
\author{Satoshi Sugiyama\thanks{sugi3@ms.u-tokyo.ac.jp}}

\date{}
\maketitle

\begin{abstract}

We prove that a positive allowable Lefschetz fibration, PALF in short, admits a structure of exact Lefschetz fibration in the sense of Seidel \cite{Se08}. If the two-fold first Chern class of the total space is zero, we obtain the Fukaya-Seidel category. We prove that the derived Fukaya-Seidel category of PALF is independent of the choice of  the symplectic structure.
At the end of this paper, we study examples and show that derived Fukaya-Seidel categories have more information than the Milnor lattices of PALFs.
\end{abstract}

\setcounter{tocdepth}{2}
\tableofcontents

\section{Introduction}\label{sec:Intro}

Our first goal in this paper is to define the Fukaya-Seidel categories for positive allowable Lefschetz fibrations, PALFs in short, and prove their invariance under the condition that the two-fold first Chern class is zero.
The Fukaya-Seidel categories are defined for exact Lefschetz fibrations, which are, roughly speaking, Lefschetz fibrations with suitable exact symplectic structure \cite{Se08}. Those categories are studied in the context of homological mirror symmetry.
The celebrated homological mirror symmetry conjecture was first proposed by Kontsevich \cite{Ko94} and predicts the equivalence of two triangulated categories $D^{\pi}Fuk(M)$ and $D^b coh(X)$ for certain pairs of Calabi-Yau manifolds $(M, X)$, called mirror pair. Here the former category $D^{\pi}Fuk(M)$ is the split closure of the  derived Fukaya category of $M$ as a symplectic manifold \cite{FOOO10} and the latter category $D^b coh (X)$ is the derived category of the category of coherent sheaves on $X$ as a complex manifold.
This conjecture is proved for several pairs of Calabi-Yau manifolds. See, for example, \cite{PZ01}, \cite{Fuk02}, \cite{Se15} and so on.

The Fukaya-Seidel categories appear when we consider the case that $X$ is a Fano manifold. In this case, the mirror partner of $X$ is a Landau-Ginzburg model $W$ \cite{HV00}.
Roughly speaking, the Landau-Ginzburg model is a holomorphic function $W$ on a K\"{a}hler manifold, called potential function, with isolated singularities.
The derived Fukaya-Seidel category $D \mathcal{F}(W)^{\to}$ is the triangulated category defined by using the data of the singularities, with the techniques of symplectic geometry \cite{Se08}, expected to be equivalent to $D^b coh(X)$.

On the other hand, PALFs are one of the most studied geometric structure in 4-dimensional topology.
The Lefschetz fibrations are completely determined in terms of monodromy operator on a regular fibre, so it is related to the study of mapping class groups of oriented surfaces, hence it has a combinatorial nature \cite{Kas80}.
If a given 4-manifold admits a structure of Lefschetz fibration, we can compute its homology groups, fundamental groups, and (some part of) the intersection forms by the data of monodromy. Moreover, if closed 4-manifold $X$ admits a structure of closed Lefschetz fibration, we can compute the signature of $X$ \cite{EN05}.

There are two very fundamental results. The first result due to Donaldson shows that every symplectic 4-manifold admits a structure of Lefschetz fibration after sufficiently many times of blow-ups \cite{Do99}. The second result due to Gompf \cite{Go05} is that every positive Lefschetz fibration admits a symplectic structure. After those two papers, there are many studies involving techniques of both symplectic geometry and PALFs, see e.g.
\cite{DS03}, \cite{Au06}, \cite{AS08}, \cite{In15}.

Along this context, the author proposes a new method to study the PALFs with symplectic technique, the derived Fukaya-Seidel
categories.
We first prove that any PALF admits a structure of exact Lefschetz fibration (Theorem \ref{thm:EExactSymplecticStr}). Thereafter, we prove that the derived Fukaya-Seidel category of a PALF is independent of the choice of the exact symplectic structure attached to the PALF (Theorem \ref{thm:InvOfFSCat}).

We can say that the concept of Fukaya-Seidel categories is a ``categorification" of the Milnor lattices since the K-groups of the derived Fukaya-Seidel categories coincide with the Milnor lattices. Thus, we naturally expect that the derived Fukaya-Seidel categories catch some sensitive information that we cannot capture it by the Milnor lattices. In Section \ref{sec:ExAndProb}, we study examples (Theorem \ref{thm:ComputationalEx}) showing that this is true, i.e. the Fukaya-Seidel categories do have more information than the Milnor lattices. In this theorem, we distinguish three PALFs that they share the same Milnor lattice by their Hochschild cohomology groups of Fukaya-Seidel categories.
Hence, we have a new method to distinguish PALFs.

The contents of this paper are as follows.
In Section \ref{sec:AlgPre1}, we review basic definitions and properties of $A_{\infty}$-categories.
In Section \ref{sec:FSCat}, we review the definition of the Fukaya-Seidel categories step by step and present a combinatorial description.
The combinatorial description is only used for the computation of examples in Section \ref{sec:ExAndProb} and \ref{sec:PrfOfCompEx}. In Section \ref{sec:LFAndExactStr}, we construct structures of exact Lefschetz fibrations to PALFs. In Section \ref{sec:InvOfFSCat}, we prove Theorem \ref{thm:InvOfFSCat}, which is the main theorem of this paper.
In Section \ref{sec:ExAndProb}, we study some examples and present a problem. In Section \ref{sec:PrfOfCompEx}, we prove some statements remained unproved in Section \ref{sec:ExAndProb}.

Throughout this paper, all manifolds are compact and have fixed orientations, the additional structures on manifolds are compatible with their orientations unless otherwise stated,
fields are algebraically closed, categories are considered over a fixed field $k$, and every hom set is of finite dimension.
We denote the closed unit disc in \(\mathbb{C}\) by \(D\), and  oriented surface with genus $g$, and $k$ boundary components by $\Sigma _{g, k}$.

{\bf Acknowledgement}

I am deeply grateful to my supervisor Toshitake Kohno for training me , giving me large numbers of pieces of advice, and encouraging me to write my master thesis.
Moreover, he also corrected numbers of errors of English of this paper. I am also grateful for it.
I would like to express my gratitude to my colleagues, especially to Shota Tonai and Tomohiro Asano for discussing with me and pointing out many errors.
In addition, I  have greatly benefitted from golden advice and fruitful discussion on proof of the main theorem with Morimichi Kawasaki and Fumihiko Sanda.
I would like to thank Noriyuki Hamada and Takahiro Ohba for pointing out a fatal error in the first version of this paper.
Finally, I want to mention that this paper can not exist without longtime support and storge of my parents.
I  express my deepest gratitude to my parents.

This work was supported by the Program for Leading Graduate 
Schools, MEXT, Japan.

\section{Algebraic preliminaries}\label{sec:AlgPre1}

In this section, we review basic definitions and properties of $A_{\infty}$-categories, just because the Fukaya-Seidel category is an $A_{\infty}$-category.
The reader who wants more detail, please refer section 1, 2 in \cite{Se08}.

\subsection{Deifnitions}\label{subsec:Def}

\begin{defn}[$A_\infty$-category]\label{def:AInfCat}
An $A_\infty$-category $\mathcal{A}$ consists
of the following data:
\begin{enumerate}
\item a set $Ob(\mathcal{A})$,
\item $\mathbb{Z}$-graded vector spaces $\displaystyle hom_{\mathcal{A}}(X, Y) = \bigoplus_{i　\in \mathbb{Z}}hom_{\mathcal{A}}^i(X, Y)$ for $X, Y \in Ob(\mathcal{A})$,
\item maps of degree $2-d$, called higher composition maps,
\begin{multline*}
\mu^d \colon hom_\mathcal{A}(X_{d-1}, X_d) \otimes hom_\mathcal{A}(X_{d-2},
X_{d-1}) \otimes \cdots \otimes hom_\mathcal{A}(X_0, X_1) 
\\ \to　hom_\mathcal{A}(X_0, X_d),
\end{multline*}
for each $d \geq 1$ and $X_0, X_1, \dots , X_d \in Ob(\mathcal{A})$.
\end{enumerate}
 These \(\mu\)'s must satisfy the \(A_\infty\)-associativity relation:
 \begin{displaymath}
\sum _{i, j, k} (-1)^{\bigstar _i} \mu ^k (a_d, \dots , a_{i+j+1}, \mu ^j
(a_{i+j}, \dots , a_{i+1}), a_i, \dots a_1) = 0,
\end{displaymath}
where $\displaystyle \bigstar _i = \sum _{1 \leq l \leq i} ( |a_l| - 1)$ , $\, (|a_l| = deg(a_l))$ and the sum is taken over all possible pairs of $i, j, k$, namely the indices run over $1 \leq j \leq d, \, k = d-j+1, \, 0 \leq i \leq d-j$.
\end{defn}

The notion of $A_{\infty}$-categories is a generalisation of dg-categories.
To see this, we first study the \(\mu\)'s and the \(A_\infty\)-associativity relations.
In the first case, when \(d=1\), the \(A_\infty\)-relation says that \(\mu ^1 (\mu
^1 (a_1)) = 0\), and \(\deg \mu ^1 = 2-1 = 1\). Thus, \((hom_{\mathcal{A}}(X_0,
X_1), \mu^1)\) is a cochain complex. We abbreviate \(\mu^1\) by \(d\) for a moment, even though we use the same letter $d$ for two meanings, degree of $\mu$'s and $\mu^1$.
Next, we consider the case that \(d=2\). We can see \(\mu^2\) as a composition of morphisms,
so we denote \(\mu^2(a_2, a_1)\) as \( (-1)^{|a_1|} a_2 \circ a_1 \). 
Then the \(A_\infty\)-relation says that \(d(a_2 \circ a_1) \pm da_2 \circ
a_1 \pm a_2 \circ da_1 = 0\). This is the (graded) Leibniz' rule between
$\mu^1$ and $\mu^2$.
In the third case, when $d=3$, the \(A_\infty\)-relation says that
\begin{displaymath}
(a_3 \circ a_2) \circ a_1 - a_3 \circ (a_2 \circ a_1) = \pm d\mu^3(a_3,
a_2, a_1) \pm \mu^3(da_3, a_2, a_1) \pm (\text{other two terms}).
\end{displaymath}
In general, the right hand side does not vanish, so the associativity
of the composition fails.
However it is homotopy associative, i.e. the induced composition maps
on the cohomology groups of morphisms are associative.

If the higher composition maps $\mu^d$ for $d > 2$ vanishes, then we obtain a non-unital dg-category by setting $\partial a \coloneqq (-1)^{|a|}\mu^1(a)$, and $a_2 \circ a_1 \coloneqq (-1)^{|a_1|}\mu^2(a_2, a_1)$.
Here, ``non-unital" means that we don't impose the condition of the existence of the identity morphisms.
Hence, the notion of $A_\infty$-categories are generalisation of dg-categories.

For an $A_{\infty}$-category $\mathcal{A}$, we define its {\it cohomology category} $H(\mathcal{A})$ by setting
$Ob(H(\mathcal{A})) \coloneqq Ob(\mathcal{A})$, $Hom_{H(\mathcal{A})}(X, Y) \coloneqq H(hom_{\mathcal{A}}(X, Y), \mu_{\mathcal{A}}^1)$ for each $X, Y \in Ob(H(\mathcal{A}))$, and the composition is defined by $a_2 \circ a_1 \coloneqq (-1)^{|a_1|}\mu_{\mathcal{A}}^2(a_2, a_1)$.
We define $H^0(\mathcal{A})$ in the same way.
The cohomology category is a graded $k$-linear category except for the possible lack of the identity morphisms.
If the cohomology category $H(\mathcal{A})$ of $\mathcal{A}$ has identity morphisms, we say that $\mathcal{A}$ is {\it cohomologically unital}, or {\it c-unital}.
In this paper, all $A_\infty$-categories are c-unital unless otherwise stated.

\subsection{Directed $A_{\infty}$-categories and its equivalences}

In this subsection, we study the concept of the directed $A_{\infty}$-categories that are mainly used in this paper.
It is worth repeating that any $A_{\infty}$-category of this paper is c-unital and the hom spaces are finite dimensional.

\begin{defn}[directed $A_{\infty}$-category]\label{def:DirAInfCat}

An $A_{\infty}$-category $\mathcal{A^{\to}}$ is a directed $A_{\infty}$-category if the following conditions hold.
\begin{enumerate}
\item The set $Ob(\mathcal{A}^{\to})$ is finite.\item There exists a total order of $Ob(\mathcal{A^{\to}})$ such that $hom_{\mathcal{A^{\to}}}(X, Y) = 0$ unless $X \leq Y$ in $Ob(\mathcal{A^{\to}})$ and $hom_{\mathcal{A^{\to}}}(X, X) = k \cdot 1_X$.
\end{enumerate}

\end{defn}

Since every totally ordered finite set $A$ is isomorphic to $\{ 1 < 2 < \cdots < n = \# A \}$ as ordered sets, for a directed $A_{\infty}$-category $\mathcal{A^{\to}}$, we set $Ob(\mathcal{A^{\to}}) = \{ X_1 < X_2 < \dots < X_n\}$ or likewise in this section.

\begin{defn}\label{def:DirSubCat}
Let $\mathcal{A}$ be an $A_{\infty}$-category and $\boldsymbol{Y} = (Y_1, Y_2, \dots Y_n)$ be a collection of objects.
We define directed $A_{¥\infty}$-category $\mathcal{A}^{\to}(\boldsymbol{Y})$ as follows:
\begin{enumerate}
\item $Ob(\mathcal{A}^{\to}(\boldsymbol{Y})) = \{ Y_1, Y_2, \dots Y_n \}$,
\item $hom_{\mathcal{A}^{\to}(\boldsymbol{Y})}(Y_i, Y_j) =
\begin{cases}
hom_{\mathcal{A}}(Y_i, Y_j) & (i < j) \\
k \cdot 1_{Y_i} & (i = j) \\
0 & (i > j) \, .
\end{cases}$
\item The higher composition maps $\{ \mu _{\mathcal{A}^{\to}(\boldsymbol{Y})}^d \}$ are induced from $\{ \mu _{\mathcal{A}}^d \}$ in the canonical way.
\end{enumerate}
\end{defn}

\begin{lem}[Lemma. 5. 21. in \cite{Se08}]\label{lem:equivOfDirSubCat}

\begin{enumerate}
\item Let $\boldsymbol{Y}$ and $\boldsymbol{Y'}$ are collections of objects in an $A_{\infty}$-category $\mathcal{A}$ satisfying each $Y_i$ and $Y'_i$ are isomorphic in $H^0(\mathcal{A})$.
Then, $\mathcal{A}^{\to}(\boldsymbol{Y})$ and $\mathcal{A}^{\to}(\boldsymbol{Y'})$ are quasi-isomorphic.
\item Let $\mathcal{F} \colon \mathcal{A} \to \mathcal{B}$ be an quasi-equivalence, $\boldsymbol{Y} = (Y_1, \dots \, , Y_n)$ be a collection of objects in $\mathcal{A}$ and $\boldsymbol{Y'} \coloneqq (\mathcal{F} Y_1, \dots \, , \mathcal{F}Y_n)$ be a collection of objects in $\mathcal{B}$. 
Then, $\mathcal{A}^{\to}(\boldsymbol{Y})$ and $\mathcal{B}^{\to}(\boldsymbol{Y'})$ are quasi-isomorphic.
\end{enumerate}

\end{lem}

\begin{rem}\label{rem:DerEquiv}
For an $A_{\infty}$-category $\mathcal{A}$, we can define the category $Tw\mathcal{A}$
of twisted complexes of $\mathcal{A}$ as a triangulated $A_{\infty}$-category.
Moreover, if an $A_\infty$-functor $\mathcal{F} \colon \mathcal{A} \to \mathcal{B}$ is quasi-equivalence,
then its induced functor $Tw\mathcal{F} \colon Tw\mathcal{A} \to Tw\mathcal{B}$
is a quasi equivalence of triangulated $A_{\infty}$-categories.

Hence, in particular, $\mathcal{F}$ induces an equivalence of derived categories
$D\mathcal{A} \coloneqq H^0Tw\mathcal{A}$ and $D\mathcal{B}$, namely $D\mathcal{F}
\colon D\mathcal{A} \to D\mathcal{B}$ is an equivalence of triangulated categories.

\end{rem}

We did not and will not go into the precise definition of the $A_\infty$-functor, quasi-equivalence of c-unital $A_\infty$-categories, and the category of twisted complexes. This is a generalisation of the construction in
the case of dg-categories\cite{BoKa91}. Details can be found in section 3
of \cite{Se08}.

\section{The Fukaya-Seidel categories}\label{sec:FSCat}

In this section, we review some basic notations and introduce the Fukaya-Seidel categories in a combinatorial way.
There is nothing new but combinatorial description of the Fukaya-Seidel categories, i.e. the sign of polygons which are used to define the higher composition maps $\mu$'s.

First, we introduce the notion of the Lefschetz fibrations, exact symplectic Lefschetz fibrations, and vanishing cycles which play an important role in this paper and introduce some basic properties of them without proofs.
Then, we define the Fukaya-Seidel categories by a combinatorial way which is originally defined via Floer theory on the vanishing cycles.

The contents of this section are just rewriting the materials in \cite{Se08}.
Hence, we can apply any theorem in \cite{Se08}.

By the way, there is a combinatorial definition of the Fukaya categories of closed Riemann surface in \cite{Ab08}.
The way of definition in this paper is very similar to that in \cite{Ab08}, even though Abouzaid treats with closed Riemann surfaces while we consider  Riemann surfaces with boundaries.
However, some signs of $A_{\infty}$-structure are not the same because the choices of Pin structures of Lagrangian branes are different.

\subsection{Exact Lefschetz Fibrations}\label{subsec:exactLF}

In this subsection, we fix some notations which we use in this paper.
We study an {\it exact symplectic manifold} $M = (M, \omega, \theta, J)$ in the sense of Seidel \cite{Se08}.
This means that $(M, \omega)$ is a sympectic manifold with corner, $\theta$ is an one form such that $d\theta = \omega$, the {\it negative Liouville vector field} $X_\theta$, which is defined by $\omega(\cdot, X_\theta) = \theta$, points strictly inwards on $\partial M$, an almost complex structure $J$ is compatible with $\omega$, and $\partial M$ is weakly $J$-convex.
We assume that the support of our Hamiltonian diffeomorphism $\varphi _H$ is away from the boundary, i.e. $\varphi_H$ fixes some open neighbourhood of $\partial M$. Set $\text{Ham}(M, \partial M)$ to be the group of such Hamiltonian diffeomorphisms. When we need to emphasise the symplectic structure, we write $\text{it as Ham}(M,
\partial M, \omega)$.

\begin{defn}[Lefschetz fibrations]\label{df:LF}

A smooth map $\pi \colon E^4 \to D$ is called a {\it Lefschetz fibration over $D$} if the following conditions are satisfied.

\begin{enumerate}
\item The total space $E$ is a manifold with corner.
\item Let $\text{Crit}( \pi )$ and $\text{Critv}(\pi )$ be the set of critical points and values of $\pi$ respectively. The map
$\text{Crit}( \pi ) \overset{\pi}\to \text{Critv}( \pi )$ is a bijection between finite sets, and $\text{Crit}( \pi ) \subset \mathring{E}$ (where $\mathring{E}$ is the set of interior points of $E$).
\item The restriction $\pi |_{\pi ^{-1} (D \setminus \text{Critv}(\pi )\, )}$ is a smooth fibre bundle, and the fiber $M$ is an oriented surface with boundary.
\item For all $p \in \text{Crit}(\pi )$, there exist complex coordinate neighbourhoods around $p$ and $\pi (p)$ such that $\pi$ is expressed as $\pi (z_0 , z_1) = z_0^2 + z_1^2$.
\item The boundary $\partial E$ satisfies the following triviality.
\end{enumerate}

The triviality condition of $\partial E$ is as follows. 
We have a natural decomposition of boundary $\partial E \cong \partial M \times D \cup \pi ^{-1}(\partial D)$ and we will identify them. The condition is that there exists an open neighbourhood $U \subset E$ of 
$\partial _h E \coloneqq \partial M \times D$ and a diffeomorphism to $\partial M \times [0, \varepsilon ) \times D$ such that the following diagram commutes.

\[\xymatrix{
U \ar[d]_{\pi} \ar[r]^{\simeq \,\,\,\,\,\,\,\,\,\,\,\,\,\,\,\,\,\,\,\,\,\,\,\,\,\,\,} & \partial M \times [0, \varepsilon ) \times D \ar[ld]^{pr_3} \\
D & \\}
\]

Here, $pr_i$ is the projection to the $i$-th component
(we will continue this notation for other projections).

For a Lefschetz fibration $\pi \colon E \to D$, 
we call $E$ the total space, $D$ the base space, and we use this notation 
$\partial _v E \coloneqq \pi ^{-1}(\partial D) \subset \partial E$.

\end{defn}

For a Lefschetz fibration $\pi \colon E \to D$ we sometimes denote it briefly by $\pi$.
Suppose we have two Lefschetz fibrations $\pi \colon E \to D$ and $\pi ' \colon E' \to D$.
An {\it isomorphism of Lefschetz fibrations} $f \colon \pi \to \pi '$ consists of two diffeomorphisms $f_{\text{tot}} \colon E \to E'$, and $f_{\text{base}}  \colon D \to D$ such that $\pi' \circ f_{\text{tot}} = f_{\text{base}} \circ \pi$.
 
\begin{rem}\label{rem:picture}

A Lefschetz fibration is a fibration with singular fibres.
The $\times$ symbol in Figure \ref{fig:LefschetzFibr} represents a critical value and the figure indicates that there is a singular point of  $A_1$-type singularity over the critical value.

\begin{figure}[hbt]
\centering
\includegraphics[width=8cm]{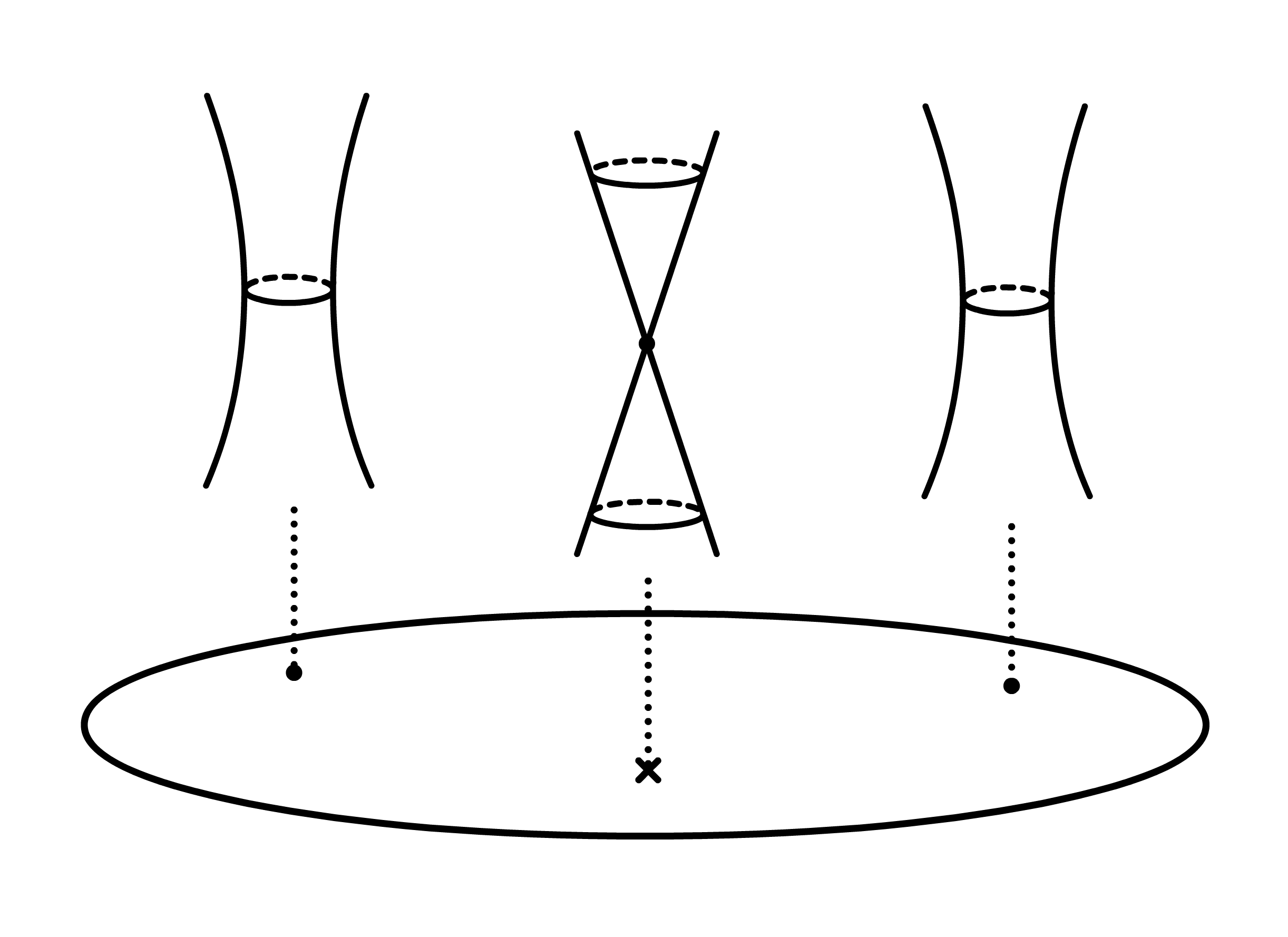}
\caption{Conceptual figure of Lefschetz fibrations}\label{fig:LefschetzFibr}
\end{figure}

\end{rem}

\begin{defn}\label{df:ExactSymplecticLefschetzFibr}

We call a map $\pi \colon E^4 \to D$ an {\it exact Lefschetz fibration} if the following conditions are satisfied.

\begin{enumerate}
\setcounter{enumi}{-1}
\item The total space $E = (E, \omega , \theta , J)$ is an exact symplectic manifold and $D = (D, j)$ is the complex closed unit disc where $j$ is a complex structure on $D$.
\item \label{enu:Jjhol} The map $\pi$ itself is a $(J, j)$-holomorphic Lefschetz fibration.
\item \label{enu:intble} The almost complex structure $J$ is integrable around $\text{Crit}(\pi )$.
\item The symplectic structure $\omega$ is the canonical one around $\partial _h E$.
\end{enumerate}

The canonical symplectic structure around $\partial _h E$ is as follows.
Let $U$ be the open neighbourhood of $\partial _h E$ that guarantees the triviality condition of Lefschetz fibration, and fix the diffeomorphism $f \colon U \to \partial M \times [0, \varepsilon ) \times D$.
Then the canonical symplectic structure is expressed as $\omega |_U = f^* (pr_{12}^* \omega _{\partial M \times [0 , \varepsilon )} + pr_3^* \omega _D )$.
Here $pr_{12}$ is the projection $pr_{12} \colon \partial M \times [0, \varepsilon ) \times D \to \partial M \times [0, \varepsilon )$, $\omega _D$ is the natural symplectic structure on $D$, and $\omega _{\partial M \times [0, \varepsilon )}$ is that of $\partial M \times [0, \varepsilon ) \cong \left( \bigsqcup S^1 \right) \times [0 , \varepsilon )$.

\end{defn}

\begin{rem}\label{rem:SeidelIsStronger}

The conditions in the definition above are stronger than those in (15a) ``Fibrations with singularities" of \cite{Se08}.
So, Seidel works in more general settings.

It is remarkable that Seidel does not assume that the symplectic form on $E$ is canonical around
$\partial _h E$. But the canonicality automatically holds when we use the definition of Seidel.
We can show this by radial trivialization of some neighbourhood of $\partial
_h E$ with the connection given by the symplectic form which we will discuss in subsection \ref{subsec:VanishingCycle}.

\end{rem}

Let $\pi \colon E \to D$ be an exact Lefschetz fibration. 
For all $y \in E$, the subspace $\ker \, (\pi _*)_y \subset T_y E$ is symplectic since 
$\pi$ is $(J, j)$-holomorphic and $J$ is $\omega$-compatible.
Hence, the regular fibre is again an exact symplectic manifold when we restrict $\omega , \theta $, and $J$ to the fibre.
We will abbreviate the restricted structures by the same symbol $\omega , \theta $, and $J$.

From now on, we use the canonical complex structure on $D \subset \mathbb{C}$ as $j$.
Hence, we will not mention the complex structure $j$ on $D$.


\subsection{Vanishing paths and vanishing cycles}\label{subsec:VanishingCycle}

We will define the Fukaya-Seidel categories of exact Lefschetz fibrations.
In this subsection, we gather the materials that we need in this paper from section 16 of \cite{Se08}. Especially, we introduce a part of the definition that we don't use the Floer theoretic method.

For an exact Lefschetz fibration $\pi \colon E \to D$, we set 
$\pi ^{\text{reg}} \coloneqq \pi |_{\pi ^{-1}(D \setminus \text{Critv}(\pi) \, )}$ and we denote the domain and target of \(\pi^{\text{reg}}\)  by $E^{\text{reg}}$ and $D^{\text{reg}}$ respectively.
For $y \in E^{\text{reg}}$, we define $H_y E^{\text{reg}} \subset T_y E^{\text{reg}}$ by 
\[H_y E^{\text{reg}} \coloneqq \left( \ker \, (\pi _*)_y \right) ^{\perp \omega } \coloneqq \{ v \in T_y E^{\text{reg}} \mid \omega (v , w) = 0 \, ( \forall w \in \ker \, (\pi _*)_y \, )\} . \]
Since, $\ker \, (\pi _*)_y$ is a symplectic subspace, we have $T_y E^{\text{reg}} = \ker \, (\pi _*)_y \oplus H_y E^{\text{reg}}$ and $\pi _* \colon H_y E^{\text{reg}} \overset{\cong}\to T_{\pi (y)} D^{\text{reg}}$.
Hence, $HE^{\text{reg}}$ defines\ a connection of $\pi ^{\text{reg}}$.
We call it the {\it symplectic connection} of $\pi ^{\text{reg}}$.

For all $y \in \partial _h E \cap E^{\text{reg}}$, we have $H_y E^{\text{reg}} \subset T_y (\partial _h E)$ by the very definition of exact Lefschetz fibration which says that the symplectic form around  $\partial _h E$ is the canonical one.
Hence, for a path $\gamma \colon [0, 1] \to D^{\text{reg}}$, we can define the parallel transport $\tilde \gamma {}_s^t \colon E_{\gamma (s)} \to E_{\gamma (t)}$.
We can prove easily that $\tilde \gamma {}_s^t$ is an isomorphism of exact symplectic manifolds (the isomorphism of exact symplectic manifolds is defined in (7a) of \cite{Se08}).

\begin{defn}\label{df:VC}

Let $\pi \colon E \to D$ be an exact Lefschetz fibration.
Fix a point $* \in D^{\text{reg}}$ and set $M \coloneqq E_*$.
We pick a path $\gamma\colon  [0, 1] \to D$ from $\text{Critv}(\pi )$ to $*$. We impose the condition that $\gamma$ is an embedding into $D$, $\gamma (0) \in \text{Critv}(\pi )$, $\gamma (1) = *$, and $\gamma ^{-1}(\text{Critv}(\pi ) \, ) = \{ 0 \}$. We call such a map $\gamma$ a {\it vanishing path}.

Let $p$ be the unique critical point of $\pi$ in $E_{\gamma (0)}$.
We define
$\Delta _{\gamma} \coloneqq \{ p \} \cup \{ y \, | \, y \in E_{\gamma (s)} \, (0 < s \leq 1), \, \lim _{t \to 0} \tilde \gamma {}_s^t(y) = p \}$ 
and we call it a {\it Lefschetz thimble}.
We say $L_{\gamma } \coloneqq \partial \Delta _{\gamma} \subset M$ a {\it vanishing cycle of $\gamma$}.

\end{defn}

\begin{figure}[hbt]
\centering
\includegraphics[width=8cm]{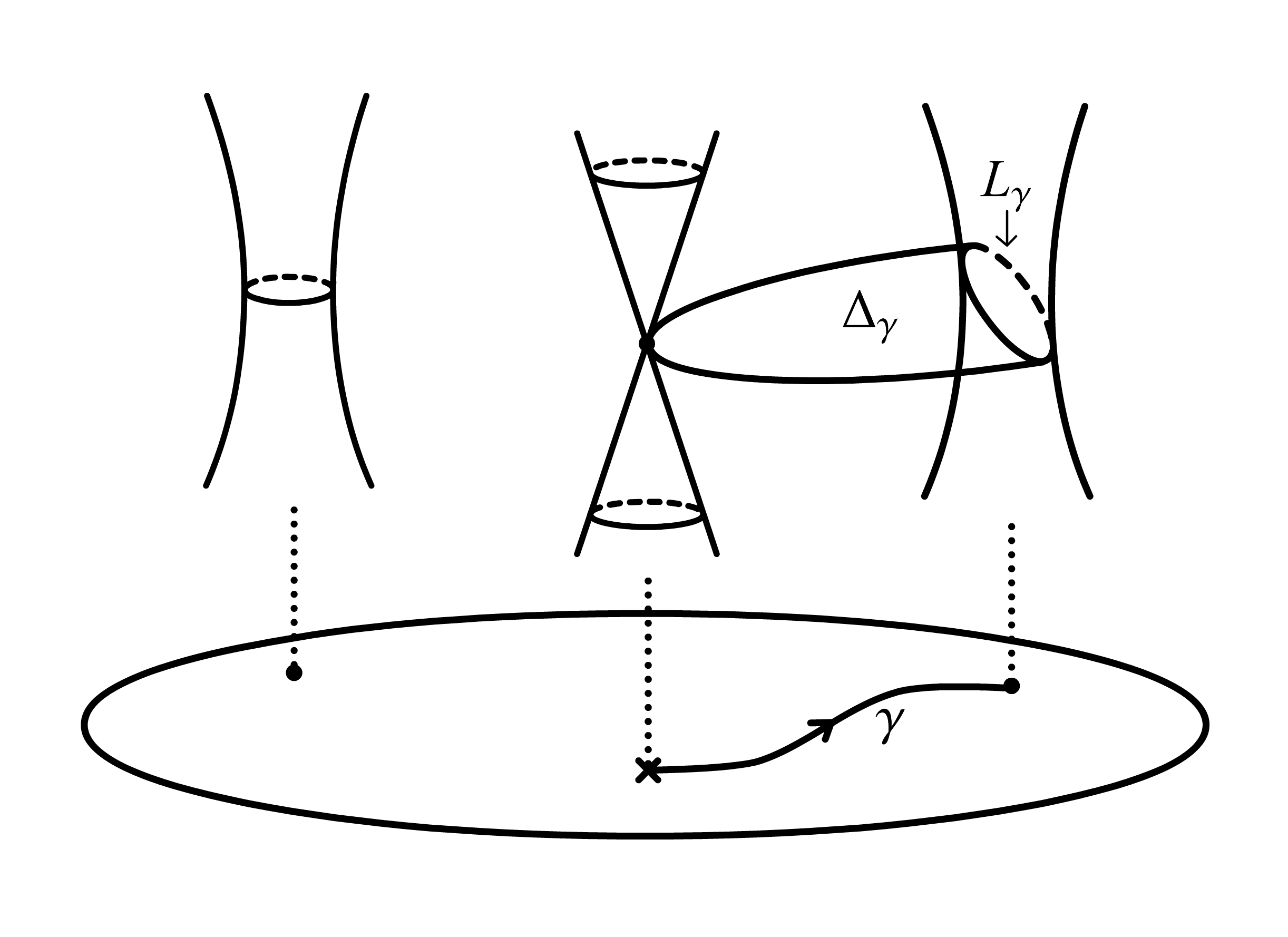}
\caption{rough picture of vanishing cycles}\label{fig:LefschetzFIbr}
\end{figure}

Thanks to the symplectic connection, we can define the concept of vanishing cycles as submanifolds of a fixed regular fibre, not as homological objects.

The limit in the definition of the Lefschetz thimble always converges.
This follows from the fact that the connection $HE$ can be definded on $E \setminus \text{Crit}(\pi )$ and 
the explicit formula of $\pi$ around the $\text{Crit}(\pi )$.
From the same reason, one can show that $\Delta _{\gamma}$ is a two-dimensional disc embedded in $E$ and $L_{\gamma}$ is an embedded $S^1$ in $M$.
For the proofs, we refer the reader to (16b) ``Vanishing cycles." of \cite{Se08}.

It is remarkable that the Lefschetz thimbles and vanishing cycles do not carry natural orientations. Hence, these two kinds of objects are the exception of the rule of this paper, ``all manifolds have fixed orientations", stated in the introduction.

\begin{rem}

We can define the vanishing cycles without the symplectic structure.
In the definition above, we defines the vanishing cycles by using symplectic form $\omega$, but what we really need is the connection on $\pi ^{\text{reg}}$ that behaves ``good" around the $\text{Crit}(\pi )$ so that $\Delta_\gamma$ and $L_\gamma$ are embedded disc and $S^1$.
Hence we can define the vanishing cycles if we have such a connection. For example take a metric $g$ on $E$ that coincides the natural metric induced by the complex coordinate around $\text{Crit}(\pi )$, and set $H_p E^{\text{reg}} \coloneqq \left( \ker \, (\pi _*)_y \right) ^{\perp g}$.
When we change the connection in a continuous manner, the  vanishing cycle moves continuously.
Moreover, since the space of connections on each vanishing path is contractible,
the free homotopy class of a vanishing cycle is independent of the choice
of the connection.
So, we use vanishing cycles for LFs to specify the singularities in the singular fibres.

As a matter of course, vanishing cycles which are defined without symplectic structure do not enjoy the symplectic properties, like Lemma \ref{lem:VCAndTheta}.

\end{rem}

\begin{lem}\label{lem:VCAndTheta}

A vanishing cycle $L_{\gamma}$ satisfies $\int _{L_{\gamma}} \theta = 0$.

\end{lem}

\begin{rem}\label{rem:VCDontHaveOri}

As we see before, there is no canonical orientation of a Lefschetz thimble
and a vanishing cycle. So, we have to give it some orientation before stating
the above lemma.
However, the orientation does not change the value of integration in this case, we didn't specify the orientation of the vanishing cycle.

\end{rem}

\begin{proof}

First, we can prove that $\Delta _{\gamma}$ is a Lagrangian submanifold
in $E$ as in (16b) ``Vanishing cycles." of \cite{Se08}.
So, we have that 
$d(\theta |_{\Delta _{\gamma}}) = \omega |_{\Delta _{\gamma}} = 0$.
By Poincar$\acute{\text{e}}$'s lemma, there exists $f \in C^{\infty}(\Delta _{\gamma})$ such that $df = \theta |_{\Delta _{\gamma}}$, hence we have $\int _{L_{\gamma}} \theta = 0$.

\hfill $\Box$

\end{proof}

We call such a Lagrangian submanifold an \textit{exact Lagrangian submanifold}.

A singular fibre of a Lefschetz fibration has a shape that the corresponding vanishing cycle is collapsed into a single point as in Figure 2.
Since $D$ is contractible, the singular fibres have whole information of a Lefschetz fibration.
Next, let us gather the information to one regular fibre $M$ to study the Lefschetz fibration.

\begin{defn}\label{df:VCs}

Fix a point $* \in \partial D \subset D^{\text{reg}}$ and set $M \coloneqq E_*$.
For all critical values, we take vanishing paths, and call them $\gamma _1 , \dots , \gamma _N$,
where $N$ is the number of the critical points, $N = \# \text{Crit}(\pi )$.
We impose the following conditions to the vanishing paths: $\gamma _i '(0) \notin T_* (\partial D)$, if $i \neq j$ then $\gamma _i([0, 1]) \cap \gamma _j ([0, 1]) = \{ * \}$, and $\gamma _i '(0)$ and $\gamma _j'(0)$ are not parallel, and $\gamma _i'(0)$'s form clockwise order. 
For such $\gamma$'s we set $\boldsymbol{\gamma} \coloneqq (\gamma _1 , \dots , \gamma _N )$ and we call it a {\it distinguished basis of vanishing paths}.
Set $L_i \coloneqq L_{\gamma _i} \subset M$, and $\boldsymbol{L}_{\boldsymbol{\gamma}} \coloneqq (L_1, \dots , L_N)$. We call $\boldsymbol{L}_{\boldsymbol{\gamma}}$ an {\it associated distinguished basis of vanishing cycles}.

\end{defn}

\begin{rem}\label{rem:AmbiguityOfVC}

Suppose that we deform the $\boldsymbol{\gamma}$ under keeping the condition that they form a distinguished basis of vanishing paths and obtain $\boldsymbol{\gamma }'$.
We write the vanishing cycles associated to $\boldsymbol{\gamma}$ by $L'_i \coloneqq L_{\gamma '_i}$.
By lemma \ref{lem:HamMove}, there exists $\phi_i \in \text{Ham}(M, \partial M)$ such that 
$\phi _i (L'_i ) = L_i$.

\end{rem}

\begin{rem}

When we change the ``isotopy class of distinguished basis of vanishing paths", the vanishing cycles is changed in terms of the Hurwicz moves, which are defined by using Dehn twists.
This is a very important consequence to prove the well-definedness of the derived Fukaya-Seidel categories. For the detail, please refer section 16 of \cite{Se08}.
\end{rem}

\subsection{Fukaya-Seidel categories and its derived categories}\label{subsec:FSCat}

We will define the Fukaya-Seidel categories step by step.
For an exact Lefschetz Fibration $\pi \colon E \to D$, to define the Fukaya-Seidel category $\mathcal{F}(\pi)^{\to}$, since it is an $A_{\infty}$-category, we have to define the set of objects, the vector spaces of morphisms, their $\mathbb{Z}$-gradings, and $\mu$'s.
We discuss them one by one.

As in the last subsection, we write a regular fibre over $* \in \partial D$ by $M = E_*$.

\subsubsection{Objects}\label{subsubsec:Obj}

Firstly,　 the objects are almost the vanishing cycles $L_1, \dots \, , L_n \subset M$ associated with some distinguished basis of vanishing cycles in Definition \ref{df:VCs}.
However, this is not a concrete definition. We need some modifications related with gradings and to obtain Lagrangian branes $L_1^{\#}, \dots , L_N^{\#}$.
The set of objects is $Ob(\mathcal{F}(\pi)^{\to}) = \{ L_1^{\#}, \dots , L_N^{\#} \}$.
In fact, we define $\mathcal{F}(\pi)^{\to}$ to be isomorphic to $Fuk(M)(L_1^{\#}, \dots , L_N^{\#})^{\to}$ with the notation in section \ref{sec:AlgPre1}.
Here $Fuk(M)$ is the Fukaya category of an exact symplectic manifold, whose definition
can be found in \cite{Se08}.
We will present the concrete definition in the following arguments.

\subsubsection{Morphisms}\label{subsubsec:Mor}

In this subsubsection, our goal is to define $hom_{\mathcal{F}(\pi)^{\to}}(L_i^{\#}, L_j^{\#})$ as a vector space (not as a graded vector space).

First, we restrict the configuration of exact Lagrangian $S^1$'s.
Let $M$ be a two dimensional exact symplectic manifold and $L_1, \dots L_N
\subset M$ be exact Lagrangian submanifolds. 
We say $\boldsymbol{L} \coloneqq (L_1, \dots , L_N )$ is in {\it general
position} if $L_i$'s are contained in $\mathring{M}$, the intersections
of $L_i$'s are all transitive, and there is no triple points. 
From now on, 
we always assume that any collection of exact Lagrangian submanifolds $\boldsymbol{L}$
in \(M\) is in general position.

\begin{defn}\label{df:PrCF}

For a collection of exact Lagrangian submanifolds $\boldsymbol{L} = (L_1,
\dots , L_N )$ we define the Floer cochain complex $CF(L_i , L_j)$ by \(\displaystyle \bigoplus _{p \in L_i \cap L_j} k \cdot [p]\), and the hom  space by
\begin{eqnarray*}
hom_{\mathcal{F}(\boldsymbol{L})^{\to}}(L_i ^{\#}, L_j^{\#})  \coloneqq
\begin{cases}
CF(L_i , L_j) & ( i < j ) \\
k \cdot e_i & ( i = j ) \\
0 & ( i > j )
\end{cases}
\end{eqnarray*}

as vector spaces, where $[p]$ and $e_i$ are formal symbols.

\end{defn}

In the sequel, we merely write $p$ for $[p]$ when no confusions can occur.

\subsubsection{Gradings}\label{subsubsec:Grading}

In this subsubsection, we give $\mathbb{Z}$-gradings to the hom spaces.

First, we prepare some concepts to define the grading.
The regular fibre $M$ of $\pi$ is diffeomorphic
to $\Sigma _{g, k}$ with some $g \geq 0$ and $k \geq 1$. We see its tangent bundle $TM$ as a complex line bundle over $M$, so
we can show that it is trivial since $H^2 (M) = 0$.
Now we fix a trivialization.
When we fix a complex structure of \(M\), there is the canonical one-to-one
correspondence between trivializations of $TM$ and non-vanishing vector fields
$X$ on $M$. 
So we now fix a complex structure of \(M\) and a non-vanishing vector field
\(X\) as a trivialization of \(TM\).

For an oriented submanifold $L \subset M$
diffeomorphic to $S^1$, we define the {\it wrihte} $w(L)$ as follows. 
Choose a map $k : S^1 \to M$ whose image coincides with $L$. 
Now we see $S^1$ as $\mathbb{R} / \mathbb{Z}$.
Then, there exists a real number $a \in \mathbb{R}$ satisfying $\mathbb{R}_{>0} \, k'(0) = exp(2 \pi ia)\mathbb{R}_{>0} \, X_{k(0)} \subset T_{K(0)}M$, and extend it to a
function $a : [0,1] \to \mathbb{R}$ satisfying that $a(0) = a$ and $\mathbb{R}_{>0}
\, k'(t) = \mathbb{R}_{>0} exp(2 \pi ia(t)) \, X_{k(t)} \subset T_{k(t)}M$.
Finally, we set $w(L) \coloneqq a(1) - a(0)$.

\begin{defn}\label{df:UnobstructedLagrangianSubmfd}

A (not necessarily oriented) submanifold $L \subset M$ diffeomorphic to $S^1$ is called {\it unobstructed with respect to a trivialization $X$}
if $w(L) = 0$ for some (hence both) orientation(s) of $L$.

\end{defn}

Now, the next lemma holds.

\begin{lem}\label{lem:ExistGoodTriv}

If the two-fold first Chern class vanishes, i.e. $2c_1(E) = 0 \in H^2(E; \mathbb{Z})$, then there exists a trivialization $X$ of $TM$ such that all vanishing cycles are unobstructed.

\end{lem}

The above lemma is discussed in (12a) and (15c) of \cite{Se08}. 
Since only the existence is essential, we won't specify the trivialization $X$. 
We fix such a trivialization $X$ and we say $L$ is unobstructed when $L$ is unobstructed with respect to $X$.

\begin{defn}[Lagrangian brane]\label{df:LagrangeBrane}

Let $L$ be an unobstructed Lagrangian submanifold.
By the unobstructedness,  there exists a function $\alpha \colon L \to \mathbb{R}$ such that $T_y L = \exp (\pi i \alpha (y) \, ) (\mathbb{R} X_y) \subset T_y M$ holds for all $y \in L$.
We call this function $\alpha$ a {\it brane structure} or {\it a grading}
of $L$. We call a triple $L^{\#} \coloneqq (L, \alpha , p)$ {\it
Lagrangian brane}, where $p$ is an arbitrary point in $L$. We call the point $p$ a {\it switching point of local trivialization of Pin structure}, and we call
$L$  the {\it underlying space} of $L^{\#}$.

\end{defn}

\begin{rem}\label{rem:AmbiguityOfBraneStr}

There are few remarks about the brane structure.
(i) Let $\alpha$ be a brane structure of $L$. For an integer $n \in \mathbb{Z}$,
we set $\alpha [n] (y) \coloneqq \alpha (y) - n$. This $\alpha [n]$ again
defines a brane structure.
This corresponds to the shift in the Fukaya category. In this manner, the brane structure has the ambiguity of $\mathbb{Z}$.
(ii) In general, a vanishing cycle doesn't have an orientation. However,
when we give it a brane structure, we can define its orientation by $\exp (\pi
i \alpha ) X$.
This orientation is called an {\it orientation of the brane}.
(iii) In the original notion of the Lagrangian brane in (12a) ``Lagrangian branes"
of \cite{Se08}, a Lagrangian brane $L^{\#}$ is defined as a triple $(L, \alpha
^{\#}, P^{\#})$ where $P^{\#}$ is a Pin structure of $L$ (in our case, it must be the non-trivial one).
The Pin structure is used to define a real line bundle over $L$ and this real
line bundle is used to define the sign in the definition of $\mu$'s.
However, to define the Fukaya-Seidel categories, what we need is just the
real line bundle. In our case, the line bundle must be the non-trivial line bundle
over $S^1$, the M\"{o}bius' band.
Furthermore, to define the categories in a combinatorial way, what we need is just fixing
a local trivialization of the line bundle. Based on the above discussion,
the point $p$ indicates that the only point which is not contained in the
region that the bundle is trivialised, i.e. we consider the trivialization
on $L \setminus \{ p \} \subset L$. Equivalently, if we go through
the point $p$, then the orientation of the fibre of the real line bundle is reversed.

\end{rem}

Now we fix brane structures $\alpha _i$ and switching points $p_i$ for $L_i$
to obtain Lagrangian branes $L_i^{\#}$.
Thus we obtain a collection of Lagrangian branes $\boldsymbol{L}^{\#} \coloneqq (L^{\#}_1, L^{\#}_2, \dots , L^{\#}_N )$.
We say $\boldsymbol{L}^{\#}$ is in {\it general position} when the underlying spaces of Lagrangian branes are in general position and $p_i$ is not contained in the underlying space of other Lagrangian branes, namely, if $i \neq j$, then $p_i \notin L_j$.
We always assume that the collection Lagrangian branes are in general position.

\begin{defn}\label{df:CF}

Let $\boldsymbol{L}^{\#} \coloneqq (L^{\#}_1, L^{\#}_2, \dots , L^{\#}_N
)$ be a collection of Lagrangian branes.
We will define $CF^*(L^{\#}_i, L^{\#}_j)$ and $hom_{\mathcal{F}(\pi )^{\to}} (L^{\#}_i, L^{\#}_j)$ as follows.

First, we define the index of $p \in L_i \cap L_j$ by $i(p) \coloneqq [ \alpha
_j (p) - \alpha _i (p) ] + 1$.
Here $[x]$ for a real number $x$ is the largest integer that is smaller than or equal to $x$.
We set $\displaystyle CF^d(L^{\#}_i, L^{\#}_j) \coloneqq \bigoplus _{p \in
L_i \cap L_j, i(p) = d} k \cdot [p]$. Finally, we define 

\begin{eqnarray*}
hom^d_{\mathcal{F}(\boldsymbol{L}^{\#})^{\to}} (L^{\#}_i, L^{\#}_j) \coloneqq 
\begin{cases}
CF^d(L^{\#}_i, L^{\#}_j) &(i < j )\\
k \cdot e_i  & (i = j \, \text{and} \, k= 0 )\\
0 &(\text{otherwise}).
\end{cases}
\end{eqnarray*}

\end{defn}

\subsubsection{$\mu$'s}\label{subsubsec:mu-s}

Next, we will define the $A_{\infty}$-higher composition maps $\{ \mu ^d \} _{d
\geq 1}$.
Our first goal is to define $\mu ^1$ and show $(\mu ^1)^2 = 0$.

We first define some moduli spaces.
Let $\Delta ^2$ denote the upper half closed unit disc i.e. $\Delta ^2 = D_+ \coloneqq \{ z \in \mathbb{C} \mid |z| \leq 1 , \text{Im}(z) \geq 0\}$.
For $i < j$ and $p, q \in L_i \cap L_j$, we define a set $\widetilde{\mathcal{M}}^2 (p ; q)$ consists of orientation preserving immersions $u \colon \Delta ^2 \to M$ satisfying the following conditions:
\begin{enumerate}
\item $u(-1) = q , u(1) = p$,
\item $u( [-1, 1] ) \subset L_i , u(S^1 \cap D_+) \subset L_j$.
\end{enumerate}
Here, $S^1$ is considered as a subset of $\mathbb{C}$, $S^1 = \{ z \in \mathbb{C}
\mid |z| = 1 \}$.

This space $\widetilde{\mathcal{M}}^2 (p ; q)$ has the natural action of
the group of the diffeomorphism of $\Delta ^2$ that fixes the corner points
and the orientation.
The quotient space of this group action will be denoted by $\mathcal{M}^2 (p ; q)$.

The moduli space $\mathcal{M}^2 (p ; q)$ becomes a set of bigons in $M$ like in Figure \ref{fig:bigon}.
Hence, $\mathcal{M}^2 (p ; q)$ is a 0-dimensional space.
For the rigorous proof we refer the reader to (13b) ``Conbinatorial Floer
cohomology" of \cite{Se08}
(Seidel writes $\widetilde{\mathcal{M}}^2$ in this paper by $\text{Imm}^2$.)
In fact, $\mathcal{M}^2(p; q)$ is a compact space, hence a finite set.
The compactness of this space is a consequence of energy estimate in (8g)
``Energy" \cite{Se08}, this is proved by the argument of Gromov's compactness.

\begin{figure}[hbt]
\centering
\includegraphics[width=8cm]{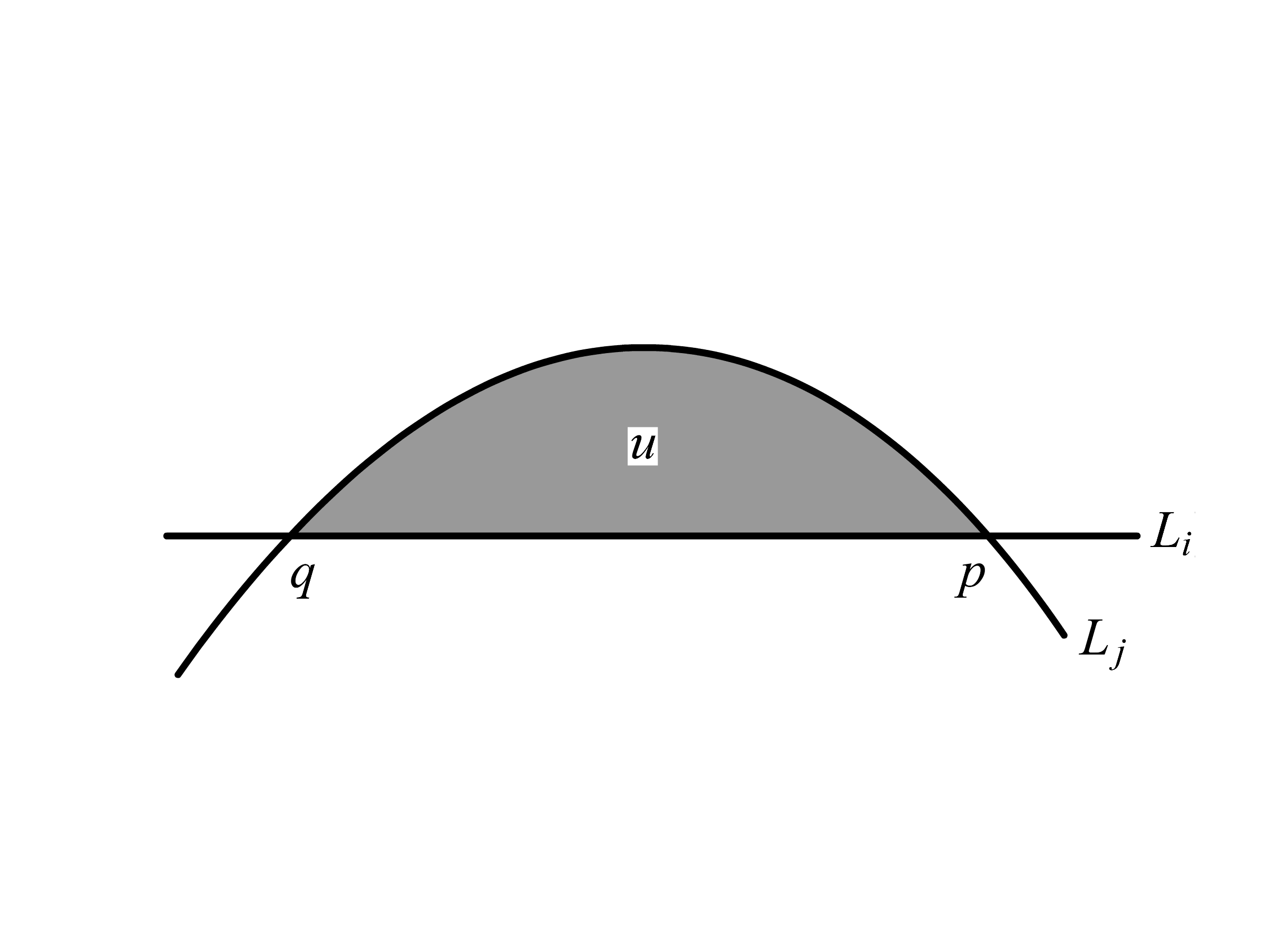}
\caption{$u \in \mathcal{M}^2(p; q)$}\label{fig:bigon}
\end{figure}

We equip a sign $(-1)^{s(u)} \in \{ \pm 1 \}$ for each $u \in \mathcal{M}^2(p; q)$ as follows.
First, we equip vertices of $u(\Delta ^2)$ with $\pm 1$. 
If the brane orientation of $L_j^{\#}$ coincides with the orientation of
that induced from $u(\partial \Delta ^2)$, then we equip them with $+1$.
If the above two orientations do not coincide, we equip $p$ with $(-1)^{i(p)}$
and $q$ with $(-1)^{i(q)}$.
Next, we equip edges with $\pm 1$. For each image of edge of $\Delta ^2$
by $u$, we equip it with $-1$ if the image contains $p_i$ or $p_j$, otherwise
we equip it with $+1$. 
Finally, we define $(-1)^{s(u)}$ by product of all the $\pm 1$ equipped to vertices
and edges.

\begin{defn}\label{df:Mu1Pr}

We define $\mu ^1 \colon CF^*(L_i^{\#}, L_j^{\#}) \to CF^*(L_i^{\#}, L_j^{\#})$ by the following formula for $i < j$.
For $p \in L_i  \cap L_j$, we set
\[\mu ^1 ([p]) \coloneqq \sum_{ \substack{q \in L_i \cap L_j \\ [u] \in \mathcal{M}^2(p; q)}} (-1)^{s(u)}  [q] .\]

\end{defn}

\begin{lem}\label{lem:mu11square}

$(\mu ^1)^2 = 0$.

\end{lem}

\begin{proof}

This statement immediately follows from the general theory of Fukaya categories of
exact symplectic manifolds in \cite{Se08}, but there is an elementary proof,
so we develop that one.
This proof is the same as that of Lemma 2.11 in \cite{Ab08} except for the sign.

In our case, we consider polygons with an obtuse angle instead of one-dimensional moduli spaces.
We take two points $p, r \in L_i \cap L_j$ and consider the coefficient of
$r$ of $\mu ^1 ( \mu ^1 (p) \, )$. 
If the coefficient is nonzero, then there must exist $q \in L_i \cap L_j$ and the
situation is, for example, like an upper-left figure of Figure \ref{fig:mu^2}.
Since $L_j$ is an embedded curve, there is a point $q' \in L_i \cap L_j$
of the upper-right figure of Figure \ref{fig:mu^2}.

\begin{figure}[hbt]
\centering
\includegraphics[width=8cm]{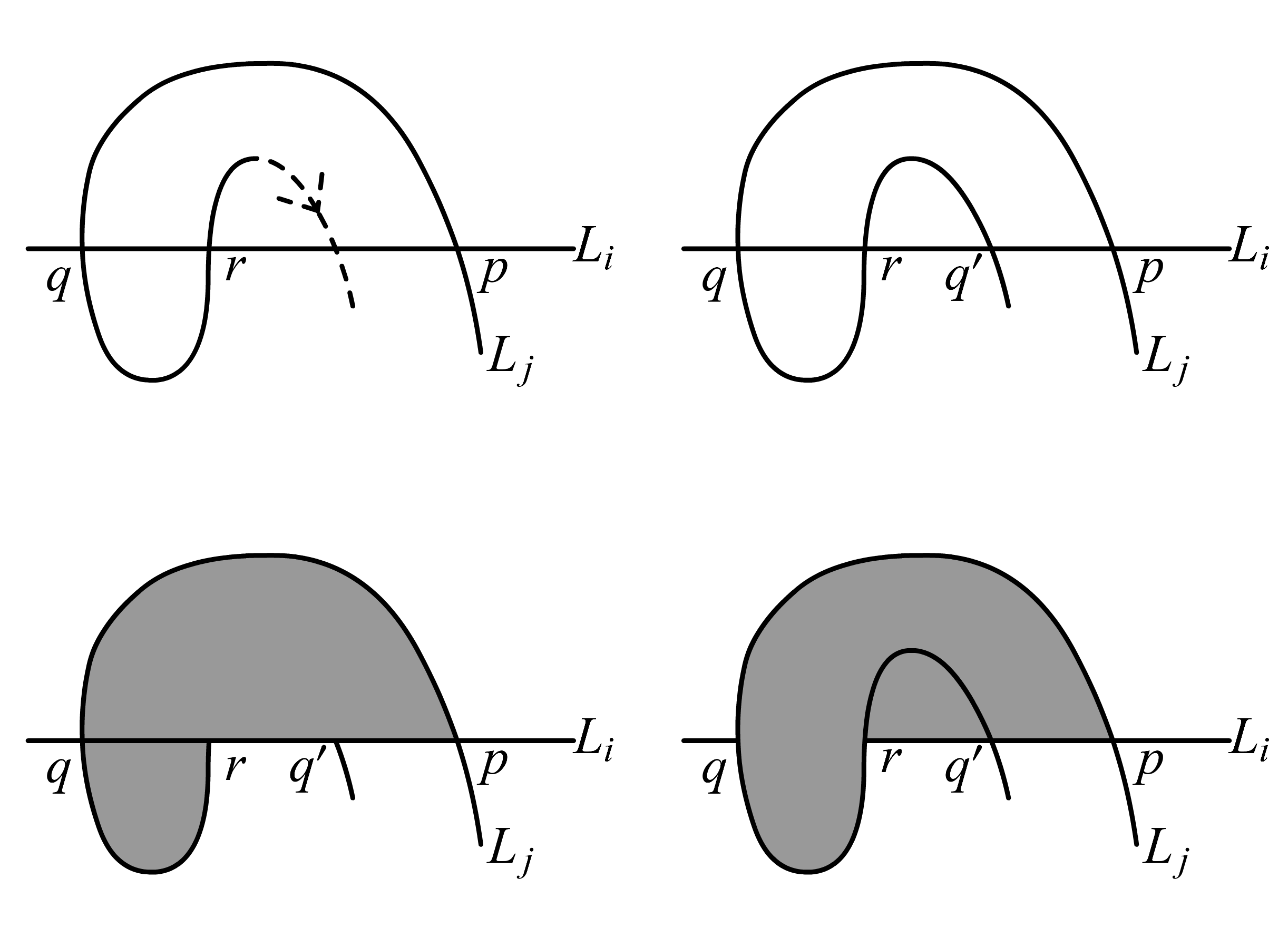}
\caption{}\label{fig:mu^2}
\end{figure}

Now, the two lower figures in Figure \ref{fig:mu^2} indicate that this region
contributes twice to the coefficient of $r$ of $\mu ^1 ( \mu ^1 (p) \, )$.
In fact, the sign of these two are different each other (we prove this later in more general settings), the contribution of Figure \ref{fig:mu^2} vanishes.
This is the case when $p, q, q', r$ are in other positions.
In this manner, immersions that contribute to the coefficient of $\mu ^1
( \mu  ^1 (p) \, )$ must occur in a pair and cancel each other, hence we obtain $(\mu ^1)^2 = 0$.
\hfill $\Box$

\end{proof}

Let us define all the other $\mu$'s.
Fix a degree $d \geq 2$. Let $i_0, i_1, \dots , i_d$ be integers satisfying $1 \leq i_0 < i_1 < \cdots
< i_d \leq N$ and 
$y_k$ a point in $L_{i_{k-1}} \cap L_{i_k}$ for $1 \leq k \leq d$, and $y_0
\in L_{i_0} \cap L_{i_d}$.
We define $\mathcal{M}^{d+1} (y_d, y_{d-1} , \dots , y_1; y_0)$ as
follows.
Let $\Delta ^{d+1}$ denote a $(d+1)$-gon. We name the vertices $v_0, v_1,
\dots , v_d$ counterclockwise, the edge connecting $v_{k-1}$ and $v_k$ by $[v_{k-1}, v_k]$ for $1 \leq k \leq d$, and the edge connecting
$v_0$ and $v_d$ by $[v_0, v_d]$.
We consider a set $\widetilde{\mathcal{M}}^{d+1} (y_d , \dots , y_1;
y_0)$ consists of orientation preserving immersions $u \colon \Delta ^{d+1}
\to M$ that satisfy the following conditions:
\begin{enumerate}
\item $u(v_k) = y_k$ for $0 \leq k \leq d$, \item $u( [v_{k-1}, v_k] ) \subset L_{i_{k-1}}$ for $1 \leq k \leq d$, and
$u( [v_0, v_d] ) \subset L_{i_d}$.
\end{enumerate}

This space $\widetilde{\mathcal{M}}^{d+1} (y_d, \dots , y_1 ; y_0)$ has the
natural action of the group of diffeomorphisms of $\Delta ^{d+1}$ that fixes
the vertices (the orientation is automatically fixed).
Let  $\mathcal{M}^{d+1} (y_d, \dots , y_1 ; y_0)$ denote the quotient space
of the group action.
Then, $\mathcal{M}^{d+1} (y_d, \dots , y_1 ; y_0)$ becomes a finite set of $(d+1)$-gons
in Figure \ref{fig:d+1gon} like the case of $\mathcal{M}^2$.

\begin{figure}[hbt]
\centering
\includegraphics[width=8cm]{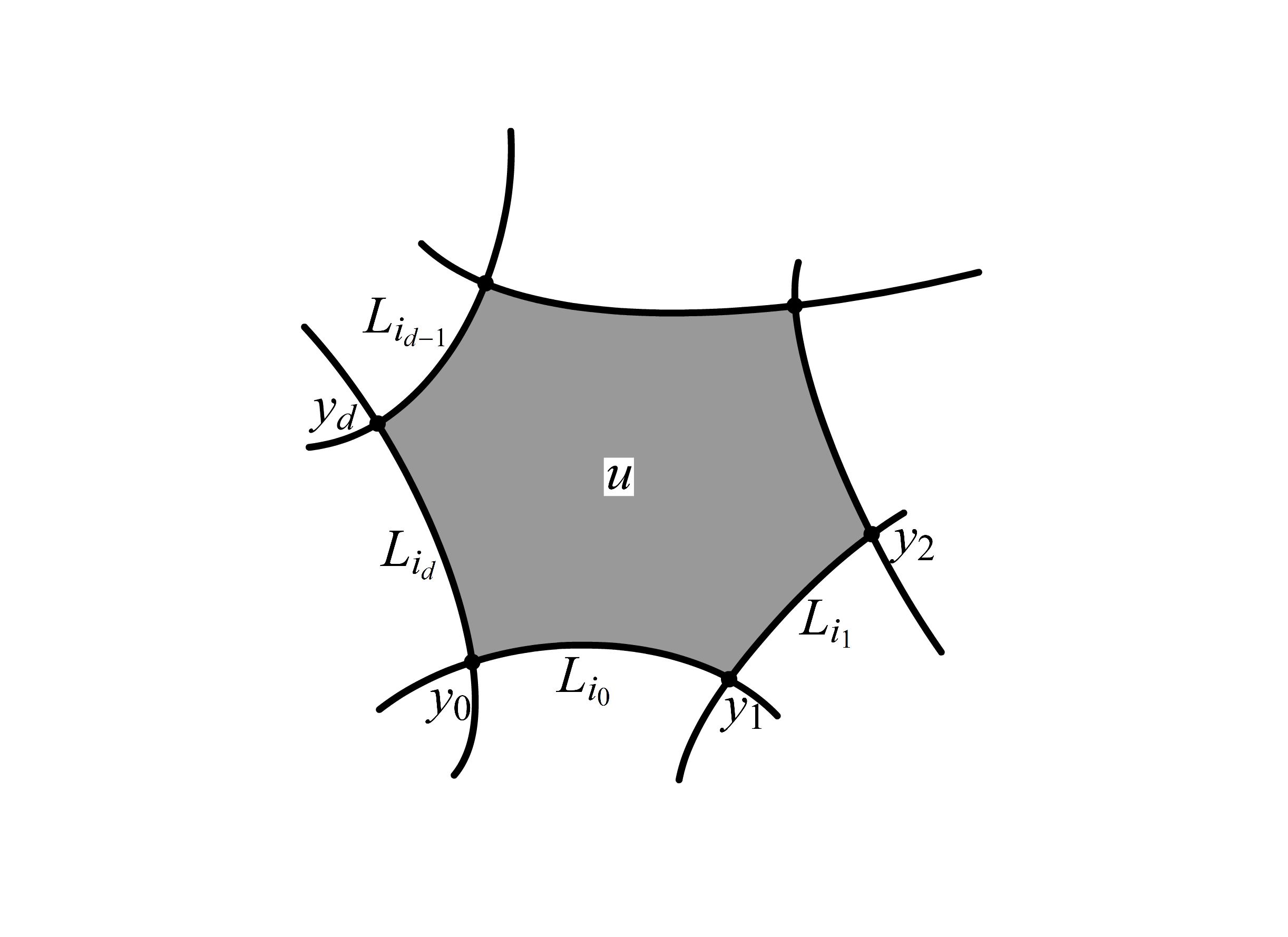}
\caption{}\label{fig:d+1gon}
\end{figure}

Next, we define the sign $(-1)^{s(u)} \in \{ \pm 1\}$ for $u \in \mathcal{M}^{d+1} (y_d, \dots , y_1 ; y_0)$.
First, we equip vertices with $\pm 1$. For a vertex $y_k$ with $1\leq k \leq
d$, we equip $y_k$ with $-1$ when the orientation of $L_{i_k}$ induced by
$\partial u$ and the orientation of the brane $L^{\#}_{i_k}$ do not coincide
and $i(y_k)$ is odd, otherwise equip $y_k$ with $+1$.
For the case of $y_0$, we equip $y_0$ with $-1$ when the orientation of $L_{i_d}$
induced by $\partial u$ and the orientation of the brane $L^{\#}_{i_d}$ do
not coincide and $i(y_0)$ is odd, otherwise equip $y_0$ with $+1$.
Next, we equip edges with $\pm 1$. For each image of the edge of $\Delta ^{d+1}$
by $u$, we equip it with $-1$ if the image contains $p_{i_k}$, otherwise we equip
it with $+1$. 
Now, we define $(-1)^{s(u)}$ by product of all the $\pm 1$ equipped to vertices
and edges.

\begin{rem}\label{rem:DifferenceBetweenAbouzaid}

The sign $(-1)^{s(u)}$ in \cite{Ab08} is just a product of $\pm 1$ corresponding
to the vertices only.
This difference comes from the different choices of the Pin structures of Lagrangian
branes, namely Abouzaid uses the trivial Pin structure of $S^1$ while
we use the non-trivial one.
The reason why we use non-trivial one is to apply Theorem 17.16 in \cite{Se08}, ``equivalence" of algebraic twist and geometric (Dehn) twist, and to obtain the well-definedness
of the derived Fukaya-Seidel categories.
More details can be fond in section 17 of \cite{Se08}.

\end{rem}

\begin{defn}\label{df:MuD}

Let $i_0, i_1, \dots , i_d$ be integers satisfying $1 \leq i_0 < i_1 < \cdots
< i_d \leq N$.
We define the operators \[\mu ^d \colon hom_{\mathcal{F}(\boldsymbol{L}^{\#})} (L_{i_{d-1}}^{\#}, L_{i_d}^{\#}) \otimes
hom_{\mathcal{F}(\boldsymbol{L}^{\#})} (L_{i_{d-2}}^{\#}, L_{i_{d-1}}^{\#}) \otimes \cdots \otimes hom_{\mathcal{F}(\boldsymbol{L}^{\#})} (L_{i_0}^{\#} , L_{i_1}^{\#})\to
hom_{\mathcal{F}(\boldsymbol{L}^{\#})} (L_{i_0}^{\#}, L_{i_d}^{\#})\]
as follows.
For $y_k \in L_{i_{k-1}} \cap L_{i_k} \, (1 \leq k \leq d)$, we set
\[\mu ^d ([y_d], [y_{d-1}], \dots , [y_1]) \coloneqq \sum_{ \substack{y_0
\in L_{i_0} \cap L_{i_d} \\ [u] \in \mathcal{M}^{d+1} (y_d, \dots , y_1 ;
y_0)}} (-1)^{s(u)}  [y_0] \, . \]

\end{defn}

\begin{rem}\label{rem:DegOfMu}

We can show that $\mathcal{M}^{d+1} (y_d, \dots , y_1 ; y_0) \neq \varnothing
\Rightarrow i(y_0) = i(y_1) + i(y_2) + \dots + i(y_d) + (2 - d)$ by the following
argument, so we can identify $\deg (\mu ^d ) = 2-d$
(this is also the case when $d = 1$).

First we consider the case when $i(y_1) = i(y_2) = \dots i(y_d)  = 1$, then
we can conclude that $\pi < \alpha _{i_d}(y_0) - \alpha _{i_0} (y_0) < 2\pi$,
hence $i(y_0) = 2$.
(To show this, one should consider the case when the image of $u \in \mathcal{M}^{d+1} (y_d, \dots , y_1 ; y_0)$ is a regular ($d+1$)-gon in $M$ and $X$ is a constant vector field around $u$ in some coordinate system. The general case follows from robustness of the indices under suitable homotopy.)
Next we can check that the value $- i(y_0) + i(y_1) + i(y_2) + \dots + i(y_d)$
is independent of the choces of the brane structures (recall that the order of subtraction
of $i(y_0)$ is different from the others).
In the case which we first consider, we have $- i(y_0) + i(y_1) + i(y_2) + \dots + i(y_d) = -2+d$, hence we can conclude that $i(y_0) = i(y_1) + i(y_2) + \dots + i(y_d)
+ (2 - d)$ holds in any case.
A more general equation is proved in lemma \ref{lem:IndexFormula}.

\end{rem}

Now, we extend the domain of $\mu$'s so that $e$'s become the identities in
our $A_{\infty}$-category.

\begin{defn}\label{df:Mu}

For $e_i \in CF^0(L^{\#}_i, L^{\#}_i)$, we define 
$\mu ^1(e_i) = 0$, 
$\mu ^2(a, e_i) = a, \mu ^2(e_i , b) = (-1) ^{|b|}b$, 
$\mu ^d ( \dots , e_i , \dots ) = 0$ for $d \geq 3$. Here $|b| = \deg (b)$ is the degree of $b$.

Furthermore, we extend the domain of $\mu$'s in the trivial way for the case
that there exists $k$ such that $i_{k-1} > i_k$.

\end{defn}

\begin{defn}\label{df:FSCat}

For a collections of Lagrangian branes $\boldsymbol{L}^{\#}$, we define 
$\mathcal{F}(\boldsymbol{L}^{\#})^{\to}$ as follows:
we set $Ob(\mathcal{F}(\boldsymbol{L}^{\#})^{\to}) \coloneqq \{ L^{\#}_1, L^{\#}_2 , \dots , L^{\#}_N \}$,
\begin{align*}hom^d_{\mathcal{F}(\boldsymbol{L}^{\#})^{\to}} (L^{\#}_i, L^{\#}_j) \coloneqq
\begin{cases}
CF^d(L^{\#}_i, L^{\#}_j) &(i < j )\\
ke_i  & (i = j \, \text{and} \, k= 0 )\\
0 &(\text{otherwise}).
\end{cases}
\end{align*}
for $L^{\#}_i, L^{\#}_j \in Ob(\mathcal{F}(\boldsymbol{L}^{\#})^{\to})$ as in Definiton \ref{df:CF}, and
\item \begin{align*}
\mu ^d_{\mathcal{F}(\boldsymbol{L}^{\#})^{\to}} \coloneqq \mu ^d \colon hom_{\mathcal{F}(\boldsymbol{L}^{\#})^{\to}}(L^{\#}_{i_{d-1}},
L^{\#}_{i_d}) \otimes hom_{\mathcal{F}(\boldsymbol{L}^{\#})^{\to}}(L^{\#}_{i_{d-2}}
, L^{\#}_{i_{d-1}}) \otimes \cdots \otimes hom_{\mathcal{F}(\boldsymbol{L}^{\#})^{\to}}(L^{\#}_{i_0},
L^{\#}_{i_1}) \\
\to hom_{\mathcal{F}(\boldsymbol{L}^{\#})^{\to}}(L^{\#}_{i_0} , L^{\#}_{i_d})
\end{align*}
as in Definition \ref{df:MuD}, \ref{df:Mu}.

\end{defn}

\begin{thm}

The above $\mathcal{F}(\boldsymbol{L}^{\#})^{\to}$ is a directed $A_{\infty}$-category.

\end{thm}

\begin{proof}

The directedness is clear from the definition, so it is sufficient to show that $\mu$'s satisfy the $A_{\infty}$-relation:
$\sum (-1)^{\bigstar _j} \mu ^l (y_d, \dots , y_{j+k+1}, \mu ^k (y_{j+k}
, \dots , y_{j+1} ) , y_j , \dots y_1) = 0$, 
where $\bigstar _j \coloneqq \sum _{1 \leq h \leq j} (i(y_h) - 1 )$.

This statement automatically follows from the general theory of the Fukaya-Seidel
categories of exact Lefschetz fibrations in \cite{Se08} like Lemma \ref{lem:mu11square}.
However, there again exists an elementary proof so we will pursue it. The following
proof is almost the same as in Lemma 3.6. in \cite{Ab08} except for  signs.

We prove the above claim only for the case $i_0 < i_1 < \cdots < i_k$ in the sense of definition \ref{df:MuD} since the other part is straight forward.
First, we note that only the polygon that has one obtuse angle like in Figure
\ref{fig:Ainfinity} contributes to the $A_{\infty}$-relation. We write it
by $w$.
This $w$ has two subdivisions $u \cup v$ and $u' \cup v'$ as in the Figure
\ref{fig:Ainfinity}.

\begin{figure}[hbt]
\centering
\includegraphics[width=8cm]{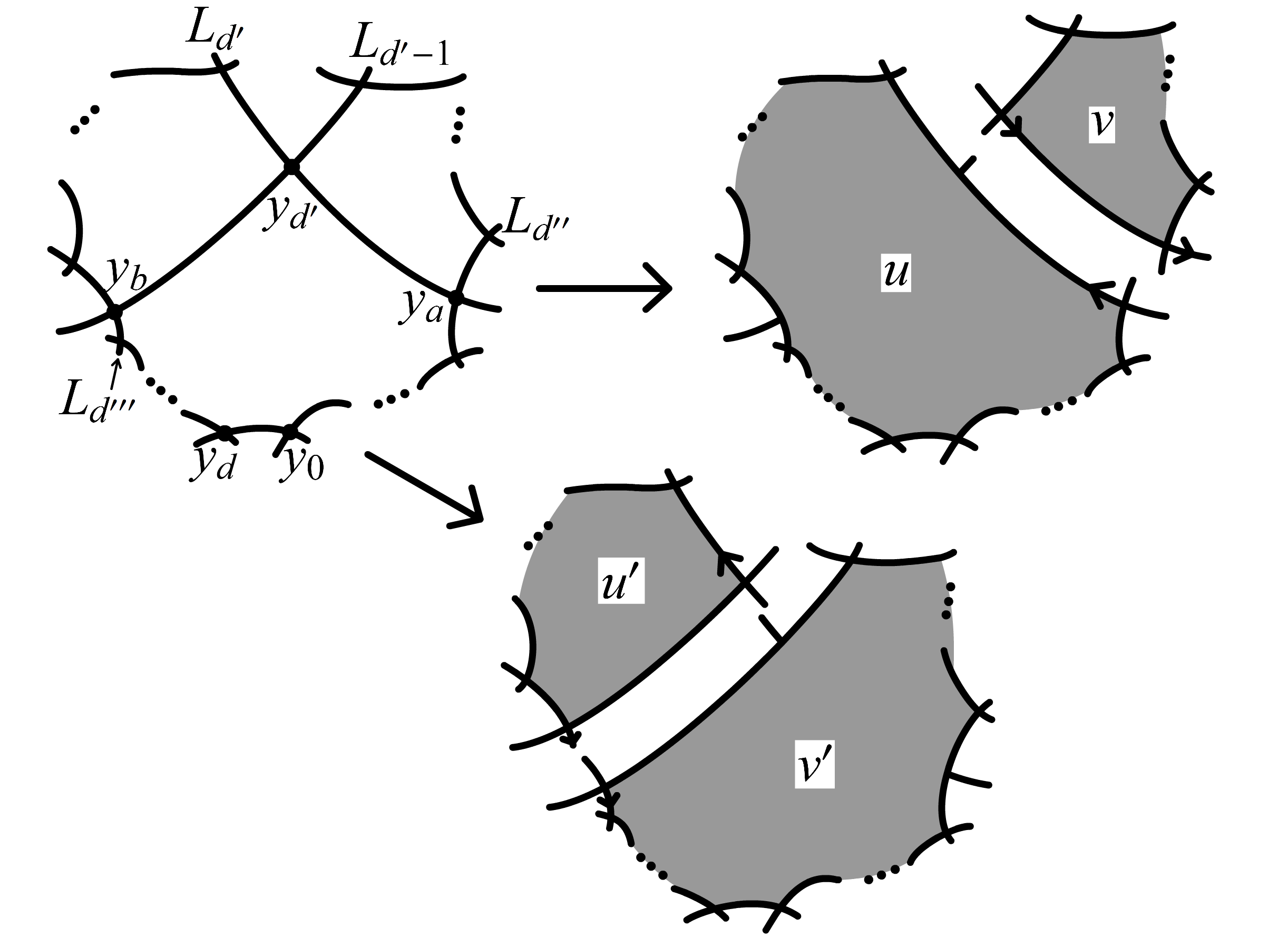}
\caption{}\label{fig:Ainfinity}
\end{figure}

Now it is sufficient to show that $ \sum _{1 \leq i \leq d''-1} (i(y_i) -
1) + s(u) + s(v) + \sum _{1 \leq i \leq d'-1} (i(y_i) - 1) + s(u') + s(v')
= 1\,  \text{mod} \, 2$, 
since the equation above shows that the terms come from $w = u \cup v$
and $w = u' \cup v'$ in the $A_{\infty}$-relation cancel each other and thus
the $A_{\infty}$-relation holds.

First, we consider how each edge contributes to $s(u) + s(v)$ and $s(u')
+ s(v')$. In both subdivision, the edges in the interior of $w$ are counted
twice or are not counted, so the interior edges don't affect the parity of
$s(u) + s(v)$ and $s(u') + s(v')$.
The contribution of edges on the boundary of $w$ is the same in both subdivisions, hence
the contribution of edges to $s(u) + s(v) + s(u') + s(v')$ is $0$ in mod
$2$.

Second, the contributions of vertices to $s(u) + s(v)$ and $s(u') + s(v')$
are the same except for $y_a, y_{d'}, y_b$.
In the case of the Figure \ref{fig:Ainfinity}, the contribution of $y_a$ and $y_{d'}$ for $s(u) + s(v)$ and for $s(u') + s(v')$ are different, so we get 
$\left( \sum _{1 \leq i \leq d''-1} (i(y_i) - 1) + s(u) + s(v)  \right) +
\left( \sum _{1 \leq i \leq d'-1} (i(y_i) - 1) + s(u') + s(v') \right) 
= \sum _{1 \leq i \leq d''-1} (i(y_i) - 1) + \sum _{1 \leq i \leq d'-1} (i(y_i)
- 1) + i(y_a) + i(y_{d'})$.

From Remark \ref{rem:DegOfMu}, we know that 
$i(y_{d'}) = i(y_a) - \sum_{d'' \leq i \leq d'-1} i(y_i) - 2 + (d' - d +
1)$ 
hence we can deduce that $\sum _{1 \leq i \leq d''-1} (i(y_i) - 1) + \sum
_{1 \leq i \leq d'-1} (i(y_i) - 1) + i(y_a) + i(y_{d'}) = 1 \, \text{mod}
\, 2$.

Even when the Figure \ref{fig:Ainfinity} is not the case, we can check in the same manner.
\hfill $\Box$

\end{proof}

\begin{defn}[Fukaya-Seidel category]\label{df:FSCatOfExactSymplecticLefschetzFibr}

For an exact Lefschetz fibration $\pi \colon E \to D$, we choose a distinguished
basis of vanishing paths $\boldsymbol{\gamma}$ and construct a distinguished basis
of vanishing cycles $\boldsymbol{L}$.
If $\boldsymbol{L}$ is not in general position, we perturb $L_i$ by some
elements in $\text{Ham} (M, \partial M)$ to make them in general position.
We also write the collection of perturbed vanishing cycles by $\boldsymbol{L}$.
Next, we give brane structures to each vanishing cycle to obtain a distinguished
basis of vanishing branes $\boldsymbol{L}^{\#} = ( L^{\#}_1 , L^{\#}_2, \dots
, L^{\#}_N )$.

Finally, we define $\mathcal{F}(\pi )^{\to} \coloneqq \mathcal{F}(\boldsymbol{L}^{\#})^{\to}$.
We call it the {\it Fukaya-Seidel category} of $\pi$.

\end{defn}

This is just a repetition but worth repeating that the definition of $\mathcal{F}(\pi)^{\to}$ above is nothing but $Fuk(M)^{\to}(\boldsymbol{L}^{\#})$ with the notation in section \ref{sec:AlgPre1}.
Here $Fuk(M)$ is the Fukaya category of an exact symplectic manifold. The definition can be found in \cite{Se08}.

\begin{thm}[(very special case of) Theorem 18.24 in \cite{Se08}]\label{thm:Invariance}

Let $\pi \colon E \to D$ be an exact Lefschetz fibration with regular fibre
$\Sigma _{g, k}$.
Then the equivalence class as triangulated $A_{\infty}$-category of the
category of the twisted complexes of the Fukaya-Seidel category $Tw \mathcal{F}(\pi
)^{\to}$ is an invariant of the exact Lefschetz fibration, i.e. it is independent
of all additional choices to define $Tw \mathcal{F}(\pi )^{\to}$.

\end{thm}

For the convenience, we recall what the additional choices are.
The additional choises are:
\begin{itemize}
\item a reference point $* \in \partial D$,
\item a distinguished basis of vanishing paths $\boldsymbol{\gamma} = (\gamma _1, \dots , \gamma _N)$,
\item Hamiltonial diffeomorphisms used to move vanishing cycles into general position,
\item trivialization $X$ of $TM$,
\item brane structures $\alpha _i$ of $L_i$,
and\item switching points of trivializations of Pin structures $p_i \in L_i$.
\end{itemize}

In this paper, we don't define the notion of the category of twisted complexes of $A_{\infty}$-categories.
For definitions and basic properties, please refer section 3 of \cite{Se08}.

Now we introduce two useful corollaries:

\begin{cor}\label{cor:InvOfDerCat}

For exact Lefschetz fibration $\pi$ as in the above theorem, the derived category of the Fukaya-Seidel category $D\mathcal{F}(\pi)^{\to} \coloneqq H^0(Tw\mathcal{F}(\pi)^{\to})$ is an invariant of $\pi$.

\end{cor}

\begin{thm}[\cite{Sh15} Theorem 4.2, Corollary 4.10]

Let $\mathcal{C}$, and $\mathcal{D}$ be c-unital $A_{\infty}$-categories.
If $Tw\mathcal{C}$ and $Tw\mathcal{D}$ are quasi equivalent, then the Hochschild cohomology  $HH^*(\mathcal{C})$ and $HH^*(\mathcal{D})$ are isomorphic as graded Lie algebras.

\end{thm}

Thanks to the above theorem, we have another corollary:

\begin{cor}\label{cor:InvOfHochschild}

For exact Lefschetz fibration $\pi$ as in theorem \ref{thm:Invariance}, the Hochschild cohomology of the Fukaya-Seidel category $HH^*(\mathcal{F}(\pi)^{\to})$ is an invariant of $\pi$.

\end{cor}

We use the Hochschild cohomology groups for distinguishing Lefschetz fibrations.
The definition of Hochschild cohomology is given in section \ref{sec:PrfOfCompEx}.

\section{Lefschetz fibrations and exact symplectic structures}\label{sec:LFAndExactStr}

In this chapter, we prove the following theorem:

\begin{thm}[Exsitence of an exact structure]\label{thm:EExactSymplecticStr}

Let $\pi \colon E^4 \to D$ be a Lefschetz fibration such that its regular fibre $M$ is diffeomorphic to $\Sigma _{g, k}$ with $k \geq 1$.
Then, the followings are equivalent:
\begin{enumerate}
\renewcommand{\labelenumi}{{\rm (\roman{enumi})}}
\item All vanishing cycles are homologically non-zero.
\item There exists $\omega , \theta , J$, and $j$ such that $\pi \colon (E^4, \omega
, \theta , J) \to (D, j)$ becomes an exact Lefschetz fibration.
\end{enumerate}

\end{thm}

If a Lefschetz fibration satisfies the condition (i) above, then it is called a PALF.

We can prove (ii)$\Rightarrow$(i) as follows.
Let $L \subset M$ be a vanishing cycle with $[L] = 0 \in H_1(M; \mathbb{Z})$, so there exists a surface $S \subset M$ such that $\partial S = L$.
Since $L$ is exact Lagrangian, we have $0 = \int _L \theta = \pm \int _S \omega \neq 0$. This is a contradiction. Hence, $L$ is homologically non-zero. 

From now, we prove (i)$\Rightarrow$(ii) in this section.
Let us fix a Lefschetz fibration $\pi \colon E^4 \to D$ such that its
regular fibre is diffeomorphic to $\Sigma _{g, k}$ and all of its vanishing
cycles are homologically non-zero.

To prove (i)$\Rightarrow$(ii), we use the following Seidel's criterion (the following statement is simplified from the original version): 

\begin{lem}[Lemma 16.9. in \cite{Se08}]\label{lem:ConstructExactLFFromVC}

Let $\boldsymbol{L} = (L_1, L_2, \dots , L_N)$ be a collection of exact Lagrangian submanifolds in $\Sigma _{g, k}$ with $k \geq 1$. Then there exists an exact Lefschetz fibration with regular fibre $\Sigma _{g, k}$, $N$ singular fibres, and distinguished basis of vanishing paths $\boldsymbol{\gamma}$ such that the corresponding distinguished basis of vanishing cycles is exactly $\boldsymbol{L}$.

\end{lem}

Hence, to prove (i)$\Rightarrow$(ii), it is enough to show the following proposition: 

\begin{prop}\label{prop:integratedValueOfThetaCanBeZero}

Let $M = (\Sigma _{g, k}, \omega , \theta, J)$ with $k \geq 1$ be an exact symplectic manifold and $L \cong S^1 \hookrightarrow M$ be an (oriented) homologically non-zero Lagrangian submanifold.
Then, there exists $\widetilde L \cong S^1 \hookrightarrow M$ such that $\widetilde L$ and $L$ are free homotopic and $\int _{\widetilde L} \theta = 0$. 

\end{prop}

To prove Proposition \ref{prop:integratedValueOfThetaCanBeZero}, we show the following two lemmas:

\begin{lem}\label{lem:RangeOfIntegratedValueOfThetaNONSEP}

If $L \subset M$ is non-separating, i.e. $M \setminus L$ is connected, then for any $a \in \mathbb{R}$, there exists $L_a$ such that $L_a$ is free homotopic to $L$ and $\int _{L_a} \theta = \int _L \theta - a$.

\end{lem}

\begin{proof}

Since $L$ is non-separating, there exists $N \cong S^1 \hookrightarrow M$ such that $N \pitchfork L$ and $L \cap N = \{ pt. \}$.
We choose a tublar neighbourhood $\iota \colon S^1 \times (-\varepsilon, \varepsilon) \hookrightarrow M$ of $N$ such that $\iota ^{-1} (L) = \{ pt. \} \times  (-\varepsilon, \varepsilon)$ and $\iota ^* \omega = d\varphi \wedge dx$, where $\varphi$ and $x$ is the canonical coordinate of $S^1$ and $(-\varepsilon,
\varepsilon)$ respectively.
We set the composition $\tilde \iota \colon \mathbb{R} \times (-\varepsilon,
\varepsilon) \to S^1 \times (-\varepsilon, \varepsilon) \hookrightarrow M$, where the first map is the universal cover $\mathbb{R} \times (-\varepsilon,
\varepsilon) \to \mathbb{R}/\mathbb{Z} \times (-\varepsilon, \varepsilon) = S^1 \times (-\varepsilon, \varepsilon).$

Let us fix a compact supported function  $h \colon (-\varepsilon, \varepsilon) \to \mathbb{R}_{\geq 0}$ on $(-\varepsilon , \varepsilon)$ such that $\int _{-\varepsilon}^{\varepsilon} h(x) \, dx = 1$.
Next, we define $L_a \hookrightarrow M$ for $a \in \mathbb{R}$ by $L_a = L \setminus \iota(\iota^{-1}(L)) \cup \tilde \iota(\text{Graph}(ah))$, where $\text{Graph}(ah) = \{ (y, x) \in \mathbb{R} \times (-\varepsilon, \varepsilon) \, | \, y = ah(x)\}$.
Then $L_a$ is well-defined as an embedded Lagrangian submanifold, and free homotopic to $L$.
Moreover, for $a \geq 0$, we have $\int _{L_a} \theta = \int _L \theta - \int _{\iota(\iota ^{-1}(L))} \theta + \int _{\tilde \iota (\text{Graph}(ah))} \theta
 = \int _L \theta - \left( \int _{\iota(\iota ^{-1}(L))}
\theta - \int _{\tilde \iota (\text{Graph}(ah))} \theta \right)
 = \int _L \theta - \int _{\{ (y, x) \in \mathbb{R} \times (-\varepsilon, \varepsilon) \, | \, 0 \leq y \leq ah(x) \}} \tilde \iota ^* \omega
 = \int _L \theta - a$. We can show that the case when $a < 0$ by almost the same argument, hence this completes the proof.
\hfill $\Box$

\end{proof}

Next, we consider the case when $L \subset M$ be separating, i.e. $M \setminus L$ is not connected.
We call the connected components of $M \setminus L$ by $M'_i$ for $i = 1,
2$, where $M'_1$ is the right part of the connected components with respect
to the orientation of $L$. We set $M_i \coloneqq M'_i \cup L \subset
M$, and $B_i \coloneqq \int _{\partial M_i \setminus L} \theta$

\begin{lem}\label{lem:RangeOfIntegratedValueOfThetaSEP}

When $L \subset M$ be separating,  for any $-B_2 < a < B_1$, there exists $L_a$ satisfying that $L_a$ is free homotopic to $L$ and $\int _{L_a} \theta = a$.

\end{lem}

\begin{proof}

Since $\int _{\partial M_1 \setminus L} \theta  - \int _L \theta = \int _{M_1} \omega > 0$, we have $\int _L \theta < B_1$.
Similarly, we have $-B_2 < \int _L \theta$.

We only prove the case that $\int _l \theta < a < B_1$ since the proof of the rest part is almost the same.
Since $0 < b \coloneqq a - \int _L \theta < B_1 - \int _L \theta$, for small $\varepsilon > 0$, there exists a collar neighbourhood $\iota \colon S^1 \times (-b - \varepsilon, 0] \hookrightarrow M_1$ such that $\iota ^* \omega = d\varphi \wedge dx$ where $\varphi$ is a natural coordinate of $S^1 = \mathbb{R} / \mathbb{Z}$ and $x$ is that of $(-b - \varepsilon, 0]$.
(To construct the collar neighbourhood, one should cut and open along some embedded intervals in $M_1$ to make $M_1$ to a cylinder $S^1 \times \left[- B_1 + \int _L \theta, 0\right]$.)
Let us set $L_a \coloneqq \iota \big(S^1 \times \{ -b\} \big)$.
We can easily check that this $L_a$ has desired properties.
\hfill $\Box$

\end{proof}

Let us prove Proposition \ref{prop:integratedValueOfThetaCanBeZero}. From the above lemmas, there exists $\widetilde L$ satisfying $\int _{\widetilde L} \theta = 0$: when $L$ is non-separating, we can find that $L_{\int _L \theta}$ can be chosen as $\widetilde L$; when $L$ is separating, it is enough to set $L_0$ as $\widetilde L$. This is possible since we have $B_i > 0$ by the definition that the negative Liouville vector field strictly inwards.

\section{Invariance of the derived Fukaya-Seidel categories}\label{sec:InvOfFSCat}

Our goal in this section is to define (Definition \ref{df:FSCat}) and   prove the invariance (Theorem \ref{thm:InvOfFSCat}) of the derived Fukaya-Seidel categories of Lefschetz fibrations.

\begin{defn}\label{df:FSCat}

Let $\pi \colon E \to D$ be a PALF such that its regular fibre
is $\Sigma _{g, k} \, ( k \geq 1)$ and the two-fold first Chern class vanishes, i.e. $2c_1(E) = 0$.
Then, by theorem \ref{thm:EExactSymplecticStr}, $\pi$ has a structure of exact Lefschetz fibration, so we fix it and think $\pi$ as an exact Lefschetz fibration.
We define the Fukaya-Seidel category $\mathcal{F}(\pi)^{\to}$ of Lefschetz
fibration $\pi$ by that of exact Lefschetz fibration $\pi$.

\end{defn}

\begin{thm}\label{thm:InvOfFSCat}

Let $\pi \colon E \to D$ be a PALF as in Definition \ref{df:FSCat}. Then the equivalence class of $Tw \mathcal{F}(\pi)^{\to}$ as triangulated $A_{\infty}$-category is an invariant of the exact Lefschetz fibration, i.e. it is independent of additional exact symplectic structure to define $Tw \mathcal{F}(\pi )^{\to}$.

\end{thm}

\begin{rem}

There are two remarkable facts. The first one is that the space of symplectic structures on $\Sigma _{g, k}$ is contractible. We can say that our proof of the above theorem is related to this fact.
The second fact is that in the case of closed Lefschetz fibrations, the space of symplectic structures suitably compatible with the fibration is contractible \cite{Go05}, where a closed Lefschetz fibration here is, roughly speaking, Lefschetz fibration over $S^2$ with closed regular fibre. If this is the case for our situation, then we can show that the space of the structures of Lefschetz fibrations is contractible.
However, even if this is the case, the above theorem will not become trivial immediately because the relative class of symplectic structure $[\omega ] \in H^2(E, \partial ^h E)$ varies and this causes different choice of the primitive $\theta$ of $\omega$.

\end{rem}

To prove the above main theorem, we use many $\varepsilon$'s in this chapter. We will always assume that all
$\varepsilon$'s are small enough, and if necessary, we replace them smaller
enough. We won't mention it to avoid unnecessary complexity.

We begin with the preparation to prove the above theorem.
We fix a PALF $\pi \colon E \to D$ in Definition \ref{df:FSCat}
and let $\omega , \theta , J$ and $\omega ', \theta ', J'$ be two different exact syplectic structures.
We write two exact Lefschetz fibrations $\pi \colon (E, \omega , \theta , J) \to D$ and $\pi \colon (E, \omega ', \theta ', J') \to D$ by $\pi$ and $\pi '$ respectively.
We will show the above theorem by proving that $Tw \mathcal{F}(\pi)^{\to}$ and $Tw \mathcal{F}(\pi ')^{\to}$ are equivalent.

Now, we fix a common based point $* \in D \setminus \text{Critv}(\pi )$ and a common distinguished basis of vanishing paths $\boldsymbol{\gamma } = ( \gamma _1 , \dots \gamma _N )$, where $N = \# \text{Crit}(\pi )$.
We set $M = M' \coloneqq E_*$ and we abbreviate the restriction of $\omega $ and $\theta $ on $M$, and $\omega '$ and $\theta '$ on $M'$ by $\omega , \theta , \omega '$, and $\theta '$.
The (almost) complex structure $J, J'$ on $M, M'$ is irrelevant to the definition of the Fukaya-Seidel categories in our situation, so we won't argue them.
We write the vanishing cycles constructed by using symplectic connection defined with $\omega , \omega '$ by $\boldsymbol{L} = (L_1, \dots L_N )$, $\boldsymbol{L}' = (L'_1, \dots L'_N)$ respectively.
Since we use the same distinguished basis of vanishing paths $\boldsymbol{\gamma }$, they are pairwise free homotopic, i.e. $[L_i] = [L'_i] \in [S^1, M]$ when we give $L_i$ and $L'_i$ appropriate orientations.

Let us write the connected component of the boundary of $M = M'$ by $C_1 = C' _1, \dots , C_k = C'_k$ respectively, and set $\displaystyle B_i = \int _{C_i} \theta , \, B'_i =
\int _{C'_i} \theta '$.
We have $B_i, B'_i > 0$ since the negative Liouville vector fields point strictly inward, but $B_i \neq B'_i$ in general.
These inequalities may cause fatal problems, so we change the identification of points in $M$ and $M'$ as follows.

\begin{lem}\label{lem:valueOfThetaOnBoundary}

To prove theorem \ref{thm:InvOfFSCat}, it is enough to prove the case when we have $B_i = B'_i$ and for some collar neighbourhoods of $C_i = C'_i$, we have $\omega = \omega '$ and $\theta = \theta '$.

\end{lem}

\begin{proof}

Let $R$ be a sufficiently large number such that for all $i$, we have $B_i + R > B'_i$.
Fix a collar neighbourhood $(-\varepsilon _1, 0] \times C_i \hookrightarrow M$. We will identify $(-\varepsilon _1, 0] \times C_i$ and its image in $M$. Now we assume that $\displaystyle \overline{\big( (-\varepsilon _1, 0] \times C_i \big)} \cap \left( \bigcup _i L_i \right) = \varnothing$ and the symplectic form can be written as  $\omega = dr \wedge d\varphi$ on the collar neighbourhood, where $r$ is the natural coordinate of $(-\varepsilon _1, 0]$ and $\varphi$ is a coordinate of $C_i$ with $\displaystyle \int _{C_i} d\varphi = 1$. 
We define $\widetilde{M}$ by gluing of $M$ and $\displaystyle \bigsqcup _i (-\varepsilon _1, R] \times C_i$ on $\displaystyle \bigsqcup_i (-\varepsilon _1, 0] \times C_i$.

Next, we extend $\omega$ and $\theta$ to $\widetilde{M}$.
We can write  $\theta = (r+B_i)d\varphi + df_i$ on the collar neighbourhood, since $\int _{C_i} \theta = B_i
= \int _{C_i} (r+B_i)d\varphi$ holds.
Extend $f_i$ to $(-\varepsilon
_1, R]$ such that $f_i(r, \varphi) = 0$ for $r > \varepsilon _1$ and set $\theta
= (r+B_i)d\varphi + df_i$.
Thus, we obtain a symplectic form on $\widetilde{M}$ which is written as $\omega = d\theta = dr \wedge d\varphi$.
Then, the negative Liouville vector field on $\{ R\} \times C_i$ is $\displaystyle -(B_i + R)\frac{\partial }{\partial r}$ hence it points strictly inwards.

Now, we repeat almost the same procedure for $M'$ but $\widetilde{M}'$ is constructed by gluing $M'$ and $(-\varepsilon _1, B_i + R - B'_i] \times C'_i$.
We identify $\widetilde{M}$ and $\widetilde{M}'$ by the diffeomorphism that is an extension of the identity map  $1_M \colon M \to M' (=M)$ and identifies $(R - \varepsilon _2, R] \times C_i$ and $(B_i + R - B'_i - \varepsilon _2, B_i + R - B'_i] \times C'_i$ in the canonical way.
By the above construction, we have $\displaystyle \int _{\{ R\} \times C_i} \theta = B_i + R = \int _{\{ B_i + R - B'_i\} \times C'_i} \theta$. Moreover, the symplectic forms and their primitives coincide on a collar neighbourhood of $\widetilde{M}$ and $\widetilde{M}'$ via the above diffeomorphism.

By the way, since the definition of the Fukaya-Seidel category is combinatorial, the Fukaya-Seidel category defined by using $L_i \subset M$ and that by using $L_i \, \subset (M \hookrightarrow ) \widetilde{M}$ are canonically isomorphic (as directed $A_\infty$-categories). Thus we have the conclusion of this lemma.

\hfill $\Box$

\end{proof}

From now on, we will prove theorem \ref{thm:InvOfFSCat} in the case of lemma \ref{lem:valueOfThetaOnBoundary}.

\begin{prop}\label{prop:VCMove}

In the setting of lemma \ref{lem:valueOfThetaOnBoundary}, there exist $\phi _1 , \dots , \phi _N \in \text{Ham}(M , \partial M, \omega)$, and
$\phi '_1, \dots , \phi '_N \in \text{Ham}(M, \partial M, \omega ')$, and $f \in \text{Diff}_0 (M, \partial M)$ such that
for all $i = 1, 2, \dots , N$, $L_i = \phi _i \circ f \circ \phi '_i (L'_i)$ and $\{ \phi '_1(L_1) , \dots , \phi '_N(L'_N) \}$ are in general position.
Here $\text{Diff}_0 (M, \partial M)$ is the identity component of the group of the diffeomorphisms of $M$ supported in  $\mathring{M}$.

\end{prop}

From the above proposition, we can derive the desired equivalence as follows.
First, we give $L_i$ and $L'_i$ brane structures those can be identified via the isotopy given by $\phi _i \circ f \circ \phi '_i$.
Then $L_i, L'_i$ have the brane orientations satisfying $[L_i] = [L'_i] \in [S^1, M]$.
Set $\phi '(\boldsymbol{L}'{}^{\# }) \coloneqq (\phi '_1 (L'_1{}^{\# }) , \dots , \phi '_N(L'_N{}^{\# }) \, )$ and 
$f \circ \phi '(\boldsymbol{L}'{}^{\# }) \coloneqq (f \circ \phi '_1 (L'_1{}^{\# }) , \dots , f \circ \phi '_N(L'_N{}^{\# }) \, )$. Then we can write 
$\boldsymbol{L}^{\# } = \phi \circ f \circ \phi '(\boldsymbol{L}'{}^{\# })$.
The following argument shows that $\boldsymbol{L}'{}^{\# }$ and $\phi ' (\boldsymbol{L}'{}^{\# })$ define the equivalent derived categories.
Since every Lagrangian brane of $\phi ' (\boldsymbol{L}'{}^{\# })$ can be moved into the corresponding Lagrangian brane of $\boldsymbol{L}'^{\#}$ by Hamiltonian diffeomorphism of $(M', \omega ', \theta ')$ respectively, they are pairwise quasi-isomorphic in $Fuk(M')$.
Recall that $\mathcal{F}(\boldsymbol{L}'^{\#})^\to = Fuk(M')^\to(\boldsymbol{L}'^{\#})$ and $\mathcal{F}(\phi'(\boldsymbol{L}'^{\#}))^\to = Fuk(M')^\to(\phi'(\boldsymbol{L}'^{\#}))$.
Hence, by lemma \ref{lem:equivOfDirSubCat}, $\mathcal{F}(\boldsymbol{L}'^{\#})^\to$ and $\mathcal{F}(\phi'(\boldsymbol{L}'^{\#}))^\to$ are quasi-isomorphic.
Similarly, $\mathcal{F}(\boldsymbol{L}^{\#})^{\to}$ and $\mathcal{F}(f \circ \phi ' (\boldsymbol{L}'{}^{\#}))^{\to}$ are quasi-isomorphic since all Lagrangian branes are  related by Hamiltonian diffeomorphism in $(M, \omega , \theta)$.
On the other hand, since the definition of Fukaya-Seidel category is combinatorial, the Fukaya-Seidel category defined by using $\phi ' (\boldsymbol{L}'{}^{\# })$ and that by using $f \circ \phi ' (\boldsymbol{L}'{}^{\# })$ are naturally isomorphic.
Hence, we have quasi-isomorphisms between four directed $A_{\infty}$-categories associated with collections of Lagrangian branes $\boldsymbol{L}^{\# }, \boldsymbol{L}'{}^{\# }, \phi ' (\boldsymbol{L}'{}^{\# }) $, and $f \circ \phi ' (\boldsymbol{L}'{}^{\# })$ thus we have the quasi-equivalences of their $Tw \mathcal{F}(-)^{\to}$ by Fact \ref{rem:DerEquiv}.

Before we start to prove Proposition \ref{prop:VCMove}, we introduce the following two lemmas:

\begin{lem}\label{lem:unbalancedWeinstein}

Suppose that our two symplectic forms $\omega $ and $\omega '$ on $M = M'$ satisfy the following condition: 
for any Lagrangian submanifold $N \simeq S^1 \hookrightarrow M$ and any of its tublar neighbourhood $\iota \colon S^1 \times (-\varepsilon _3, \varepsilon _3) \hookrightarrow M$ such that $\iota ^* \omega ' = d\varphi \wedge dx$, the equality $\displaystyle \int _{S^1 \times (-\varepsilon _3, 0)} \iota ^* \omega \, = \int _{S^1 \times (0, \varepsilon _3)} \iota ^* \omega$ holds. Here $\varphi$ is a coordinate
of $S^1$ and $x$ is a coordinate of $(-\varepsilon _3, \varepsilon _3)$.
Then, there exists $k>0$ such that $\omega ' = k \omega$.

\end{lem}

The proof of this lemma is elementary so we omit it.
In fact, the stronger version also holds, namely, the hypothesis of the lemma can be changed from ``for any $N$" to ``for one $N$".

\begin{lem}\label{lem:HamMove}

Let $M \coloneqq (\Sigma _{g, k} , \omega , \theta)$ be a two dimensional exact symplectic manifold and $L_1 ,
L_2 \subset \mathring{M}$ be its Lagrangian submanifolds diffeomorphic to $S^1$.
Suppose that $L_1$ and $L_2$ are isotopic each other, both homologically non-zero, and $\int _{L_1} \theta = \int _{L_2} \theta$.
Then, there exists $\phi \in \text{Ham}(M , \partial M , \omega)$ such that $\phi (L_2) = L_1$.

\end{lem}

This lemma is a well-known result, so we omit the proof too.

Now, we begin the proof of Proposition \ref{prop:VCMove}.
We set $E_i \coloneqq \int _{L'_i} \theta$. Recall that those values are not zero in general.
We show it by bearing it into two cases.

First, we consider the case that there does not exist $k>0$ such that $\omega ' = k \omega$.
We set $f$ to be the identity map $f = 1_M$.
We will show it by  induction on $i$, namely we will prove step by step that there exist $\phi _j \in \text{Ham}(M, \partial M, \omega)$ and $\phi '_j \in \text{Ham}(M, \partial M, \omega ')$ for $1 \leq j \leq i$ such that $L_j = \phi _j \circ \phi '_j (L'_j)$ and $\{ \phi '_1 (L'_1) , \dots \phi '_i(L'_i) , L'_{i+1} , \dots , L'_N\}$ are in general position.
Now, we construct such $\phi_i$ and $\phi'_i$ under the hypothesis that we already have such $\phi_j, \phi'_j$ for $j < i$. It is enough to show that there exists $\phi '_i \in \text{Ham}(M, \partial M, \omega ')$ satisfying the conditions of general position and $\int _{\phi '_i (L'_i)} \theta = 0$,
since we can find $\phi _i \in \text{Ham}(M, \partial M, \omega)$ such that $L_i = \phi _i \circ \phi '_i (L'_i)$ by lemma \ref{lem:HamMove}.

In the case when $E_i = 0$, it is enough to  set $\phi '_i = 1_M$.
In this case, the condition of general position automatically holds.
Now we consider the case $E_i \neq 0$.
We again bear it in the following two cases: the case that $L'_i$ is non-separating and the case that $L'_i$ is separating.

In the first case, choose a Lagrangian $S^1$, $N \hookrightarrow M$, such that $N \pitchfork L$ and $N \cap L = \{ pt. \}$.
From (the strong version of) Lemma \ref{lem:unbalancedWeinstein} we can find  a tublar neighbourhood $\iota \colon S^1 \times (-\varepsilon _4, \varepsilon _4) \to M$ of $N$ such that $\iota ^* \omega ' = d\varphi \wedge dx$, and $\int _{S^1 \times (-\varepsilon _4 , 0)} \iota ^* \omega \, \neq \int _{S^1 \times (0, \varepsilon _4)} \iota ^* \omega$.
We can find such a tublar neighbourhood among the neighbourhoods satisfying that $\iota ^{-1}(L) = \{ 0 \} \times (-\varepsilon_4, \varepsilon_4)$, so we assume this condition.
We set $A \coloneqq \int _{S^1 \times (0, \varepsilon _4)} \iota ^* \omega \, - \int _{S^1 \times (-\varepsilon _4, 0)} \iota ^* \omega \, \neq 0$.

Now, we take $0 < \varepsilon _5 \ll \varepsilon_4$ and a function $\overline{h} \colon (0 , \varepsilon _4) \to \mathbb{R}_{\geq 0}$ such that $\lim _{t \to +0} \overline{h}^{(k)}(t) = \lim _{t \to +0} \overline{h}^{(k)}(\varepsilon _{4}- t) = 0$ for any $k \geq 0$,
$\overline{h}|_{(\varepsilon _5, \varepsilon _4- \varepsilon _5)} = 1$, $\overline{h}'|_{(0, \varepsilon _5)} > 0$, and
$\overline{h}'|_{(\varepsilon _4- \varepsilon _5, \varepsilon _4)} < 0$.

Using this function, we define a smooth function $h \colon (-\varepsilon _4 , \varepsilon _4) \to \mathbb{R}$ by $h(x) = \text{sign}(x) h(|x|)$.
Here, $\text{sign}(x)$ is a sign function, that is defined to be $\pm 1$ corresponding to the sign of $x$ and  $0$ if $x = 0$.
Then, we consider the following Hamiltonian $H_i$.

\begin{eqnarray*}
H_i(p) = 
\begin{cases}
0 & p \notin \iota ( S^1 \times (-\varepsilon _4, \varepsilon _4) \,   \\
\int _{-\varepsilon _4}^x h(y) \, dy & p = \iota (s, x)
\end{cases}
\end{eqnarray*}

Now we write the Hamiltonian diffeomorphism generated by $H_i$ with respect to $\omega '$ by $\phi '_{H_i}$.
Then, we have $\int _{\phi '_{H_i}(L'_i)} \theta \, - \int _{L'_i} \theta \, 
= \int _{\{ (s, x) | s = h(x)\} \subset S^1 \times (-\varepsilon _4, \varepsilon _4)} \iota ^* \theta \, - \int _{\{ (s, x) | s = 0 \} \subset S^1 \times (-\varepsilon _4, \varepsilon _4)} \iota ^* \theta \, 
= - \int _{\{ (s, x) | 0 \leq s \leq h(x) , 0 \leq x \leq \varepsilon _4\}} \tilde \iota {}^* \omega \, +\int _{\{ (s, x) | h(x) \leq s \leq 0 , -\varepsilon _4 \leq x \leq 0 \}} \tilde \iota {}^* \omega$.
Here, $\tilde \iota$ is the composition $\mathbb{R} \times (-\varepsilon _4, \varepsilon _4) \to S^1 \times (-\varepsilon _4, \varepsilon _4) \xrightarrow{\iota} M$.
This value can be arbitrary close to $-A$ by getting $\varepsilon _5$ small.
Likewise, for the case of $\phi '_{H_i}{}^n \, ( = \phi '_{nH_i} )$, the value $\int _{\phi '_{H_i}{}^n(L'_i)} \theta \, - \int _{L'_i} \theta$ can be arbitrary close to $-nA$ by getting $\varepsilon _5$ small.

Now, for example, we consider the case $A, E_i >0$.
If $n$ is large enough and $\varepsilon _5$ is small enough, then $\int _{\phi '_{H_i}{}^n (L'_i)}\theta = \int _{\phi'_{nH_i} (L'_i)}\theta$ is negative.
Hence, by the intermediate value theorem, there exists $a_i \in (0, n)$ such that $\int _{\phi '_{a_i H_i}(L'_i)} \theta \, = 0$.

Finally, we only have to check the condition of general position.
When the above $\phi '_{a_i H_i}$ does not achieve the condition of general position, we can achieve the condition by modifying the tubular neighbourhood and the function $\overline{h}$.
In the case that the signs of $A$ and $E_i$ are different, we can prove in the same way.

Next, we consider the second case that $L'_i$ is separating.
If $E_i = 0$, it is enough to set $\phi '_i = 1_M$ as in the case that $L'_i$ is non separating, so we consider the case $E_i \neq 0$.
First, we name the connected components of $M \setminus L'_i$. The one which is located the right side of $L'_i$ with respect to the orientation of $L'_i$ and $M$ is denoted by $M_1$ and the other one by $M_2$.
Set $S_j \coloneqq \partial M \cap M_j$ for $j = 1, 2$.
When we set $A_j \coloneqq \int _{S_j} \theta \, \left( = \int _{S_j} \theta ' \right)$, we have $\int _{M_1} \omega ' = A_1, \int _{M_2} \omega ' = A_2, \int _{M_1} \omega = A_1 - E_i, \int _{M_2} \omega = A_2 + E_i$ by the Stokes' theorem.

In the following discussion, we only consider the case $E_i > 0$.
Define a function $g$ on $M$ by $\omega - \omega ' = g\omega '$.
Then we have $\int _{M_1} g\omega ' = -E_i, \int _{M_2} g\omega ' = E_i$. Set $U_1 \coloneqq \{ p \in M_1 | g(p) < 0\}$ and $U_2 \coloneqq \{ p \in M_2 | g(p) > 0\} \neq \varnothing$. Then we have $\int _{U_1} g\omega ' \leq - E_i, \int _{U_2} g\omega ' \geq E_i$.
Let us consider the region $K_1 \coloneqq \{ p \in M_1 | g(p) \leq -\varepsilon _6\}$, and $K_2 \coloneqq \{ p \in M_2
| g(p) \geq \varepsilon _6 \}$.
If we choose $\varepsilon _6$ to be so small and $\pm \varepsilon _6$ be the regular values of $f$, 
then we have $\int _{K_1} g\omega ' \leq - (1 - \varepsilon _7) E_i, \int
_{K_2} g\omega ' \geq (1 - \varepsilon _7) E_i$ for some small $\varepsilon_7 > 0$, and $K_j$'s are compact two dimensional manifolds with boundary with some connected components.
Before proceeding the discussion, we prove the following lemma.

\begin{lem}

The ratio $\varepsilon_7 / \varepsilon_6$ can arbitrarily be small.

\end{lem}

\begin{proof}

By the definition of $\varepsilon_7$, it is enough to satisfy these inequalities $\int _{K_1} g\omega
' \, -(-E_i)\leq \varepsilon _7 E_i$ and $\int_{K_2} g\omega ' \, - E_i
\geq - \varepsilon _7 E_i$. Consider the following inequality 
\begin{align*}
\int _{K_1} g\omega' \, -(-E_i) &\leq \int_{K_1} g\omega' - \int_{U_1} g\omega'
\\
&= \int_{U_1 \setminus K_1} -g\omega' \\
&= \int_{-\varepsilon_6 < g(p) < 0, p \in M_1} -g\omega' \\
&\leq \varepsilon_6 \int_{-\varepsilon_6 < g(p) < 0, p \in M_1} \omega',
\end{align*}
and 
\begin{equation*}
\int_{K_2} g\omega ' \, - E_i \geq - \varepsilon_6 \int_{0 < g(p) < \varepsilon_6,
p \in M_2} \omega'.
\end{equation*}
So, $\varepsilon_7$ is enough to satisfy that $\displaystyle \varepsilon_6  \cdot
\text{Max} \left\{ \int_{-\varepsilon_6 < g(p) < 0, p \in M_1} \omega'  \,
, \int_{0 < g(p) < \varepsilon_6, p \in M_2} \omega' \right\} \leq \varepsilon_7
E_i$.
Now, the above two integrals in $\text{Max}$ converge to $0$ when
$\varepsilon_6 \to 0$. Thus the ratio $\varepsilon_7 / \varepsilon_6$ can arbitrarily be small.
\hfill $\Box$

\end{proof}

Set the connected component decomposition $K_j = \bigsqcup _{k} K_{jk}$.
There exist closed sets $F_{jk} \subset \mathring{K_{jk}}$ such that $F_{jk}$ is diffeomorphic to $D$ and $\left| \int _{F_{jk}} g\omega ' \, \right| \geq (1 - \varepsilon _7) \left| \int _{K_{jk}} g\omega
' \, \right|$ holds.
Then, we have $\sum _k \int _{F_{1k}} g \omega ' \, \leq -(1 - 2 \varepsilon _7) E_i$ and $\sum _k \int _{F_{2k}} g \omega ' \, \geq (1 - 2 \varepsilon _7) E_i$.

Consider the deformation of $L'_i$ presented as follows.
First, we choose a pairwise disjoint paths $\gamma _{jk} \colon [0, 1] \to M$ such that $\gamma _{jk}((0, 1]) \subset M_j, \, \gamma _{jk}^{-1} (L'_i) = \{ 0 \} , \, \dot \gamma _{jk} (0) \notin T_{\gamma _{jk}(0)}L'_i , \, \gamma _{jk} ^{-1} (F_{jk}) = \{ 1 \}, \, \dot \gamma _{jk} (1) \notin T_{\gamma _{jk}(1)}(\partial F_{jk})$, and $\gamma _{jk} ^{-1} (F_{j'k'}) = \varnothing$ if $(j', k') \neq (j, k)$ as in upper part of Figure \ref{fig:BlowingBalloon}. 
Such paths do exist. The reason is as follows.First,  $M \setminus \left( \bigsqcup F_{jk} \right)$ is connected. When we have already found a part of $\gamma$'s, the complement of $F$'s and $\gamma$'s $M \setminus \left( \bigsqcup F_{jk} \cup \bigsqcup (\text{image of already found } \gamma \text{'s})\right)$ is again connected. Thus we can construct them inductively.
Now, we deform $L'_i$ like ``blowing balloons" along the $\gamma$'s and $F$'s as in Figure \ref{fig:BlowingBalloon}.

\begin{figure}[hbt]
\centering
\includegraphics[width=8cm]{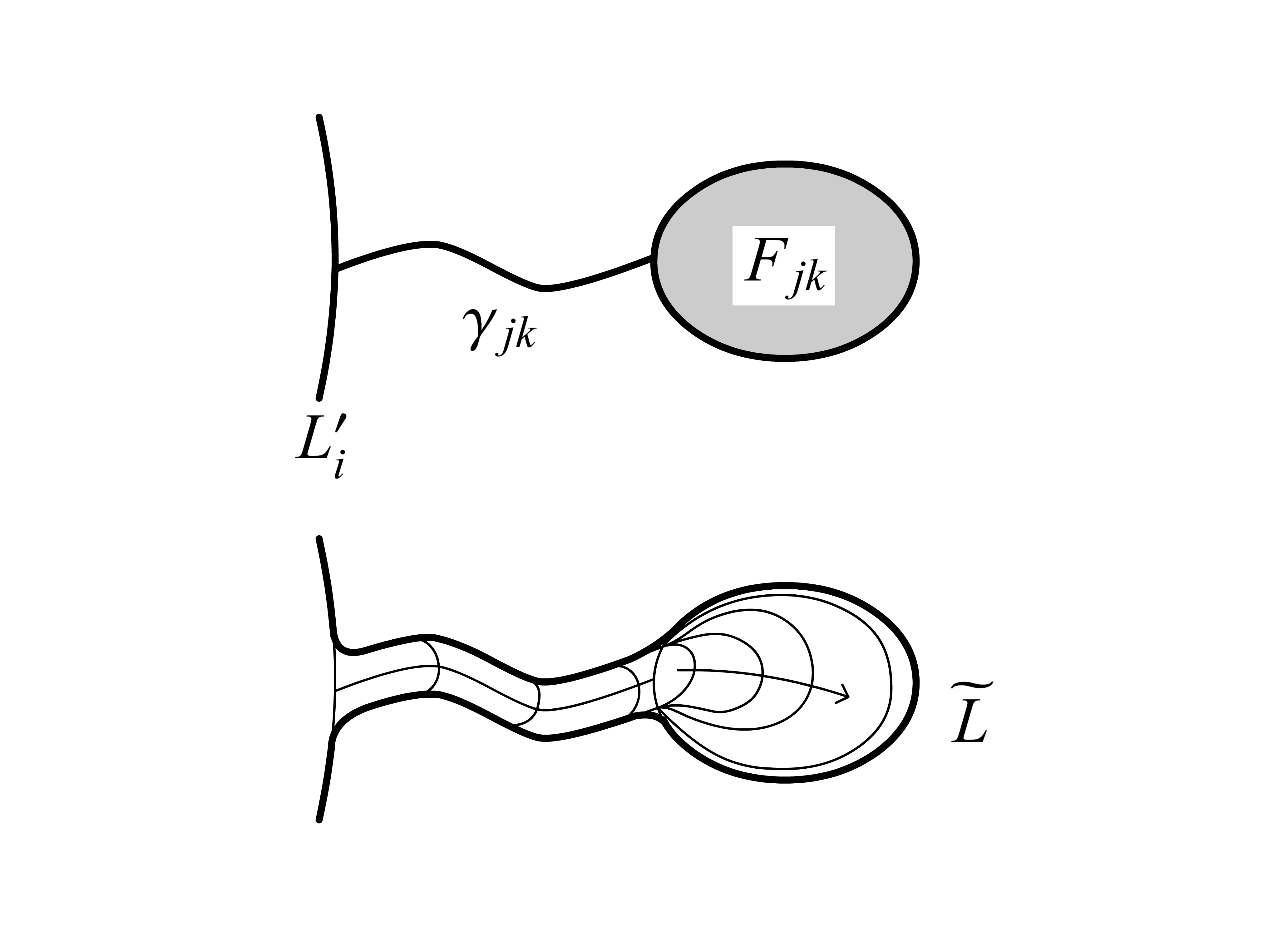}
\caption{}\label{fig:BlowingBalloon}
\end{figure}

In the process of the deformation, we impose a condition that the integrated value of $\theta '$ on deformed $L'_i$ must sustain zero.
If $\sum _k \int _{F_{1k}} \omega ' \, < \sum _k \int _{F_{2k}} \omega '$, then the balloons in $M_1$ become full before those in $M_2$ are not full yet. We stop the blowing at that time (we do the same when the inequality with the other direction holds).
We name the resulting curve $\widetilde L$. Let us estimate the value $\int _{\widetilde L} \theta$.
We define $D_{jk}$ by the region surrounded by $L'_i$ and $\widetilde L$ and intersects with $F_{jk}$.
Then we have 
\begin{align*}
\int _{\widetilde L} \theta
 &= \int _{L'_i} \theta + \sum _k \int _{D_{1k}} g\omega ' - \sum _k \int _{D_{2k}} g\omega ' \\
 &= E_i + \sum _k \left( \int _{F_{1k}} g\omega ' + \int _{D_{1k} \setminus F_{1k}} g\omega '\right) - \sum _k \left( \int _{F_{2k} \cap D_{2k}} g\omega ' + \int _{D_{2k} \setminus F_{2k}} g\omega ' \right) \\
 & \leq E_i - (1 - 2\varepsilon _7)E_i - \sum _k \int _{F_{2k} \cap D_{2k}} g\omega ' - \sum _{jk} (-1)^j \int _{D_{jk} \setminus F_{1k}} g\omega ' \\
 &\leq 2 \varepsilon _7 E_i - \varepsilon _6 \sum _k \int _{F_{2k} \cap D_{2k}} \omega ' - \sum _{jk} (-1)^j \int _{D_{jk} \setminus F_{1k}} g\omega '.
\end{align*}

We will show that the value of the last term of the above inequality can be negative.
If we make $D_{jk}$ around $\gamma _{jk}$    thin enough, then we have
$F_{2k} \cap D_{2k} \neq \varnothing$, so we assume this.
When the ratio $\varepsilon_7 / \varepsilon_6 \geq \left( \sum _k \int _{F_{2k} \cap D_{2k}} \omega ' \right) / 2E_i$, we choose $\varepsilon^{\text{new}}_6$ and $\varepsilon^{\text{new}}_7$ smaller than $\varepsilon_6$ and $\varepsilon_7$ respectively such that $\varepsilon^{\text{new}}_7 / \varepsilon^{\text{new}}_6 < \left( \sum _k \int _{F_{2k} \cap D_{2k}} \omega ' \right) / 2E_i$.
Since we change the value of $\varepsilon_6$ and $\varepsilon_7$, the subsets $F_{jk}$ and $D_{jk}$ must be changed into $F^{\text{new}}_{jk}$ and $D^{\text{new}}_{jk}$. We can impose that $\bigcup_{j,k} F_{jk} \subset \bigcup_{j,k'} F^{\text{new}}_{jk'}$ and $\bigcup_{j,k} D_{jk} \subset \bigcup_{j,k'} D^{\text{new}}_{jk'}$ by the very construction, so we assume these. Finally, we have 
\begin{align*}
\frac{\varepsilon^{\text{new}}_7}{\varepsilon^{\text{new}}_6} <  \frac{ \sum _k \int _{F_{2k} \cap D_{2k}} \omega '}{2E_i} \leq  \frac{ \sum _k \int _{F^{\text{new}}_{2k'} \cap D^{\text{new}}_{2k'}} \omega ' }{2E_i},
\end{align*}
hence we have $\displaystyle 2 \varepsilon^{\text{new}}_7 E_i - \varepsilon^{\text{new}}_6 \sum _k \int _{F^{\text{new}}_{2k} \cap D^{\text{new}}_{2k}} \omega ' < 0$.

From now, we write such  $\varepsilon^{\text{new}}_6, \varepsilon^{\text{new}}_7, F^{\text{new}}_{jk'}$, and $D^{\text{new}}_{jk'}$ by $\varepsilon_6, \varepsilon_7, F_{jk}$, and $D_{jk}$.
If we get $D_{jk}$ around $\gamma _{jk}$  thinner and thinner, the value $\left| \sum _{jk} \, (-1)^j \int _{D_{jk} \setminus F_{1k}} g\omega ' \right|$ is getting smaller and smaller,　 and the coefficient $\sum _k \int _{F_{2k} \cap D_{2k}} \omega '$ of $\varepsilon_6$ is getting bigger and bigger.
Hence we have $\int _{\widetilde L} \theta < 0$ when $D_{jk}$ around $\gamma _{jk}$ is thin enough.

When we use the ``blowing balloons" process with the above $\varepsilon_6, \varepsilon_7, F_{jk}$, and $D_{jk}$, there exists $L''_i$ in the middle of the process such that $\int _{L''_i} \theta = \int _{L''_i} \theta ' = 0$ from the intermediate value theorem.
Thus, by lemma \ref{lem:HamMove}, there exists $\phi '_i \in \text{Ham} (M, \partial M)$ such that $\phi '_i(L'_i ) = L''_i$.
If the transversality condition fails, we deform the shape of $D_{jk}$ and blowing balloon process to achieve it and do the same process.

In the case that there exists a constant $k$ such that $\omega ' = k \omega$ we choose $f$ from $\text{Diff}_0(M, \partial M) \setminus \text{Symp}_0(M, \partial M, \omega)$.
Then our task is to find $L''_i$ isotopic to $L'_i$ such that $\int _{L''_i} \theta ' = 0$ and $\int _{f(L''_i)} \theta = 0$.
This is equivalent to consider the case that $f = 1_M$ and $\omega, \theta$ is replaced by $f^* \omega, f^* \theta$.
Since $f$ does not preserves $\omega$, there does not exist $k>0$ such that $\omega '= k f^* \omega$.
This case is already proved.

\begin{rem}\label{rem:NeedPrevLemma}

Proposition \ref{prop:VCMove} cannot be proved by the same argument when we don't assume
the result in Lemma \ref{lem:valueOfThetaOnBoundary}.
For example, let $M = M'$ be $S^1 \times [0, 1]$, $\varphi$ and $r$ be the canonical coordinates for $S^1$ and $[0, 1]$ respectively. Set $\omega = \omega ' = d\varphi \wedge dr$, $\theta = -\left(r - \frac12 \right)d\varphi$,
$\theta ' = -\left(r - \frac13 \right)d\varphi$, $L = S^1 \times \left \{
\frac12 \right \} \subset M$, and $L' = S^1 \times \left \{ \frac13 \right \}
\subset M'$.
Then, $L$ divides $M$ into two components of the same area while $L'$ divides $M'$ into two component of the different areas. Since every Hamiltonian
diffeomorphism preserves the area, we have to pick
$f \in \text{Diff}_0 (M, \partial M)$ very specifically, while in the above proof,
$f$ is just used in order to break the relation $\omega' = k\omega$.

\end{rem}

\section{Examples and problems}\label{sec:ExAndProb}

\subsection{K-groups of the Fukaya-Seidel categories and the Milnor lattices}\label{subsec:FSCatAndMilnorLattice}

In this section, we present examples showing that the derived
Fukaya-Seidel categories have more information than the Milnor lattices.
Before we go to the calculation of the examples, we review fundamental
features of the Fukaya-Seidel categories.

First, we recall some definitions.
Let \(\pi\) be a Lefschetz fibration. 
We fix a regular fibre \(M\), and gather vanishing cycles to \(M\), namely
\(L_1, L_2, \dots , L_N\), where \(N = \# Crit(\pi)\).
Then, we define a free \(\mathbb{Z}\)-module by \(M_{\pi} := \mathbb{Z}L_1
\oplus \mathbb{Z}L_2 \oplus \cdots \oplus \mathbb{Z}L_N\) and consider a
pairing that is induced from the intersection pairing.
We call the pair, \(M_{\pi}\) and its pairing, the {\it Milnor lattice}.
This is the fundamental fact that the isomorphism class as free \(\mathbb{Z}\)-module
with the pairing of the Milnor lattice of a Lefschetz fibration $\pi$ is independent
of the choice of the additional geometric data, the choice of regular fibre, the connection of $\pi^{\text{reg}}$,  and the distinguished basis of vanishing paths.

In the case of PALFs with vanishing two-fold first Chern class, the Milnor lattice is reconstructed from the derived Fukaya-Seidel category $\mathcal{F}(\pi)^\to$ as follows. First, we prepare some definitions. For a triangulated category $\mathcal{T}$, we define its {\it K-group} $K(\mathcal{T})$
as quotient group $\bigoplus _{X \in Ob(\mathcal{T})} \mathbb{Z}X / \sim$.
Here $\sim$ is generated by $X -Y +\ Z ~\sim 0$ for exact triangle $X \to
Y \to Z$.
If the following hom space ${\rm Hom}_\mathcal{T}^*(X, Y) \coloneqq  \bigoplus _{i \in \mathbb{Z}} {\rm Hom} (X, T^i Y)$ is of finite dimension for any pair of objects, then we can define the {\it Euler pairing} on the K-group by
\[\bigg( [X], [Y] \bigg) := \sum_{d \in \mathbb{Z}} (-1)^d \left( \dim _{k}
{\rm Hom}(X, T^d Y) - \dim _{k} {\rm Hom}(T^d Y, X) \right),\]
for $[X], [Y] \in K(\mathcal{T})$.
The Euler pairing is well-defined, i.e. this is independent of the choice of the representative $X, Y$ in the RHS of the definition.

Finally, for an exact Lefschetz fibration with vanishing
two-fold first Chern class $\pi$, its Milnor lattice and the
K-group of derived Fukaya-Seidel category with the Euler pairing\(\) are naturally isomorphic
as free \(\mathbb{Z}\)-module with a pairing.

By this fact, one sometimes says that the notion of the Fukaya-Seidel categories is a categorification
of the Milnor lattices.
So we can expect that the Fukaya-Seidel categories have much information
than the Milnor lattices.
In the next subsection, we give examples which show that the Fukaya-Seidel categories do have much information than the Milnor lattices.

\subsection{Examples and problems}

In this subsection, we consider three Lefschetz fibrations $\pi _1,
\pi _2$, and $\pi _3$ with regular fibre \(\Sigma _{3, 1}\).
Those Lefschetz fibrations are defined by specifying their vanishing
cycles  \(L_1, L_2, \text{ and } L_3\) as in Figure \ref{fig:LFs}.

\begin{figure}[hbt]
\centering
\includegraphics[width=8cm]{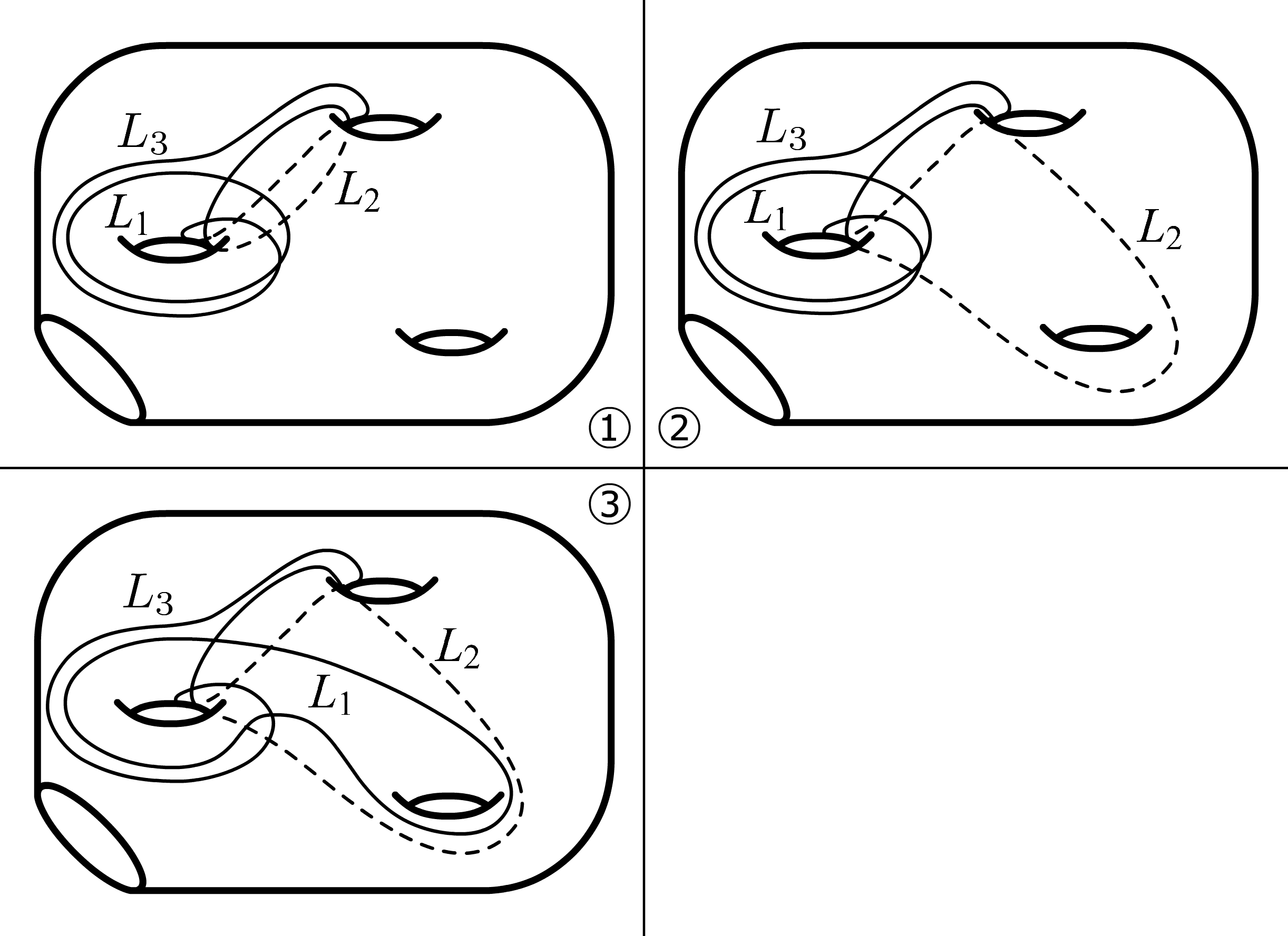}
\caption{Vanishing cycles of $\pi _i$}\label{fig:LFs}
\end{figure}

These three Lefschetz fibrations have isomorphic Milnor lattices, namely $M_{\pi_i} \cong \mathbb{Z}^3$, $([L_j],[L_j]) = 0$, and $([L_j], [L_k]) = \pm 1$ for $j \neq k$ (the sign depends on the orientations of the vanishing cycles which we have not defined yet).
The vanishing cycles of the first Lefschetz fibration $\pi_1$ enclose two triangles, one is (relatively) small and located on the ``front side" of the regular fibre $\Sigma_{3, 1}$, while the other may be hard to find, a grey shaded triangle in Figure \ref{fig:pi1}.
In the second Lefschetz fibration, a vanishing cycle $L_2$ is set to be different from that of $\pi_1$ so that the ``grey triangle" no longer appears, and in the third Lefschetz fibration, $L_1$ is also changed in order to break both of triangles. Hence, the $A_\infty$-structures of the Fukaya-Seidel categories of these three Lefschetz fibrations differ from each other, and in fact, their derived categories are also not equivalent.
This is what we will see in this and the next section.

\begin{lem}\label{lem:c1Vanish}

The above three PALFs $\pi _i \colon E_i \to D$ satisfy the condition of vanishing of the two-fold first Chern class, i.e. $2c_1(E_i) = 0$.

\end{lem}

\begin{proof}

It is known that the total space of a Lefschetz fibration with $N$ critical
points are homotopy equivalent to a topological space which is obtained by attaching
$N$ discs to its regular fibre along its vanishing cycles \cite{Kas80}.
By using Mayer-Vietris' exact sequence, we can compute their homology groups:
$H_1(E_2; \mathbb{Z}) \cong \mathbb{Z}^3,
H_1(E_3; \mathbb{Z}) \cong \mathbb{Z}^3 \oplus \mathbb{Z} / 2\mathbb{Z}$, and $H_2(E_2; \mathbb{Z}) = H_2(E_3; \mathbb{Z}) = 0$.
Hence, we can conclude $H^2(E_2; \mathbb{Z}) = 0$ and $H^2(E_3; \mathbb{Z}) = 0 \text{ or } \mathbb{Z}/2$, thus $2c_1(E_1) = 0$, and $2c_1(E_3) = 0$.

In the case $i = 1$, two homology classes $[L_1], [L_2]$ are linearly independent and $L_3 \approx \tau _{L_1} (L_2)$.
Now, $M \cong \Sigma _{3, 1}$ is deffeomorphic to a plumbing of six $S^1 \times [-1, 1]$ and we can find a diffeomorphism that sends $L_1$, $L_2$ to two (distinct) $S^1 \times \{ 0 \} \subset S^1 \times [-1, 1]$ of six.
Hence we can find a trivialization $X$ such that $w(L_1) = w(L_2) = 0$.
Since Dehn twist preserves unobstructed exact Lagrangian submanifold \cite{Se00}, we have $w(L_3) = 0$.
Now, we can conclude that $2c_1(E_1) = 0$ as follows.
Our trivialization naturally gives a non-vanishing section $\eta _M^2$ of $(T^*M)^{\otimes 2}$ (where the tensor product is taken in the complex sense).
By the discussion in (16f) ``Grading issues" in \cite{Se08}, $E_1$ admits a relative quadraic complex volume form $\eta _{E_1/D}^2$, which is a non-vanishing section of complex line bundle $\mathcal{K}_{E_1/D}^2 \coloneqq (\pi _1)^* (TD)^{\otimes 2} \otimes \wedge ^{top} (TE_1)^{\otimes (-2)}$ (this complex line bundle is appeared in (15c) ``Relative quadratic volume form" of \cite{Se08}).
Since the complex line bundle $TD$ is trivial, we can conclude $\wedge ^{top} (TE_1)^{\otimes (-2)}$ is trivial. Hence we have $2c_1(E_1) = 0$.
\hfill $\Box$

\end{proof}

From the above lemma, we can define the Fukaya-Seidel category of $\pi_i$.
These three PALFs can be distinguished by Hochschild cohomology groups as follows.

\begin{thm}\label{thm:ComputationalEx}

The above three Lefschetz fibrations can be distinguished in terms of Fukaya-Seidel categories. Namely, their Hochschild cohomology groups are as follows: 

\begin{align*}
HH^0(\mathcal{F}(\pi _1)^{\to}) &\cong k, &
HH^1(\mathcal{F}(\pi _1)^{\to}) &\cong k, \\
HH^0(\mathcal{F}(\pi _2)^{\to}) &\cong k, &
HH^1(\mathcal{F}(\pi _2)^{\to}) &\cong 0, \\
HH^0(\mathcal{F}(\pi _3)^{\to}) &\cong k^2, &
HH^1(\mathcal{F}(\pi _3)^{\to}) &\cong k.
\end{align*}

\end{thm}

The proof of this theorem is presented in the next section.

The Milnor lattices of the above three Lefschetz fibrations all agree, so
this is an example that the Fukaya-Seidel categories do have more information
than the Milnor lattices.
But in fact, the total space of $\pi _1, \pi _2, \pi_3$ are not homeomorphic each other, so there leave a lot to be desired. One can prove this by computing their first homology groups $H_1(E_1; \mathbb{Z}) \cong \mathbb{Z}^4, H_1(E_2; \mathbb{Z}) \cong \mathbb{Z}^3$, and $H_1(E_3; \mathbb{Z}) \cong \mathbb{Z}^3 \oplus \mathbb{Z} / 2\mathbb{Z}$ by the Mayer-Vietris' exact sequence.
So, there emerges a natural question:

\begin{prob}\label{prob:FSCatHaveMoreInfo}

Is there two PALFs \(\pi _1, \pi _2\) with vanishing of the two-fold first Chern class such that the total
spaces are homeomorphic (or diffeomorphic), the Milnor lattices are isomorphic,
but their category of twisted complexes of Fukaya-Seidel categories are not equivalent?

\end{prob}

The category of twisted complexes has more
information than the derived category of a given $A_{\infty}$-category.
Namely, Kajiura \cite{Kaj13} proposed two $A_{\infty}$-categories $\mathcal{C}_0$ and
$\mathcal{C}_1$ in such that $D \mathcal{C}_0 \cong D \mathcal{C}_1$
but $Tw \mathcal{C}_0 \ncong Tw \mathcal{C}_1$. So there emerges another question:

\begin{prob}\label{prob:GeometricExOfKaj13}

Is there a geometric example of this? This asks that whether there exist two Lefschetz fibrations \(\pi _1, \pi _2\) such that $D \mathcal{F}(\pi_1
)^{\to} \cong D \mathcal{F}(\pi_2 )^{\to}$
but $Tw \mathcal{F}(\pi_1 )^{\to} \ncong Tw \mathcal{F}(\pi_2 )^{\to}$.

\end{prob}

\section{Proof of theorem \ref{thm:ComputationalEx}}\label{sec:PrfOfCompEx}

\subsection{Computation of the derived Fukaya-Seidel categories}

Let us calculate the derived Fukaya-Seidel categories for $\pi _1, \pi _2$, and $\pi _3$ in section \ref{sec:ExAndProb}.
Now, we can choose $L_1, L_2$, and $L_3$ in Figure \ref{fig:LFs} as vanishing
cycles which means that there exists $\theta$ such that $\int_{L_i} \theta = 0$.
For $\pi_2$ and $\pi_3$, we can find such $\theta$ by adding (representatives of) some elements of $H^1_{\text{dR}}(E_i; \mathbb{R})$.
In the case of $\pi_1$, we should use a symplectic form $\omega$ such that two triangles have the same area and its primitive  $\theta$ satisfying $\int_{L_1} \theta =\ \int _{L_2} \theta = 0$.  Then, $\int_{L_3} \theta = 0$ automatically holds. 

We name the distinguished basis of vanishing cycles of $\pi _i$ by $\boldsymbol{L}_i = (L_1, L_2, L_3)$.
Here, $L_1$ out of $\boldsymbol{L}_1$ is the vanishing cycle of $\pi _1$ and $L_1$ out of $\boldsymbol{L}_2$ is the vanishing cycle of $\pi _2$ etc. We use the same symbol $L_1$ for the first vanishing cycle of $\pi _1, \pi_2$, and $\pi_3$.

We write the Lagrangian branes corresponding to $L_1, L_2$, and $L_3$ by $L_1^{\#} = (L_1, \alpha _1, p_1), L_2^{\#}$, and $L_3^{\#}$ and name the intersection points by $\{ p_{ij} \} = L_i \cap L_j$ for $i < j$.
Now we will compute their Fukaya-Seidel categories defined via the above Lagrangian branes $\mathcal{F}(\boldsymbol{L}_i^{\#})^{\to}$ one by one.

\subsubsection{Computation of $\mathcal{F}(\boldsymbol{L}_1^{\#})^{\to}$}\label{subsubsec:CompOfFL1}

\begin{figure}[hbt]
\centering
\includegraphics[width=5cm]{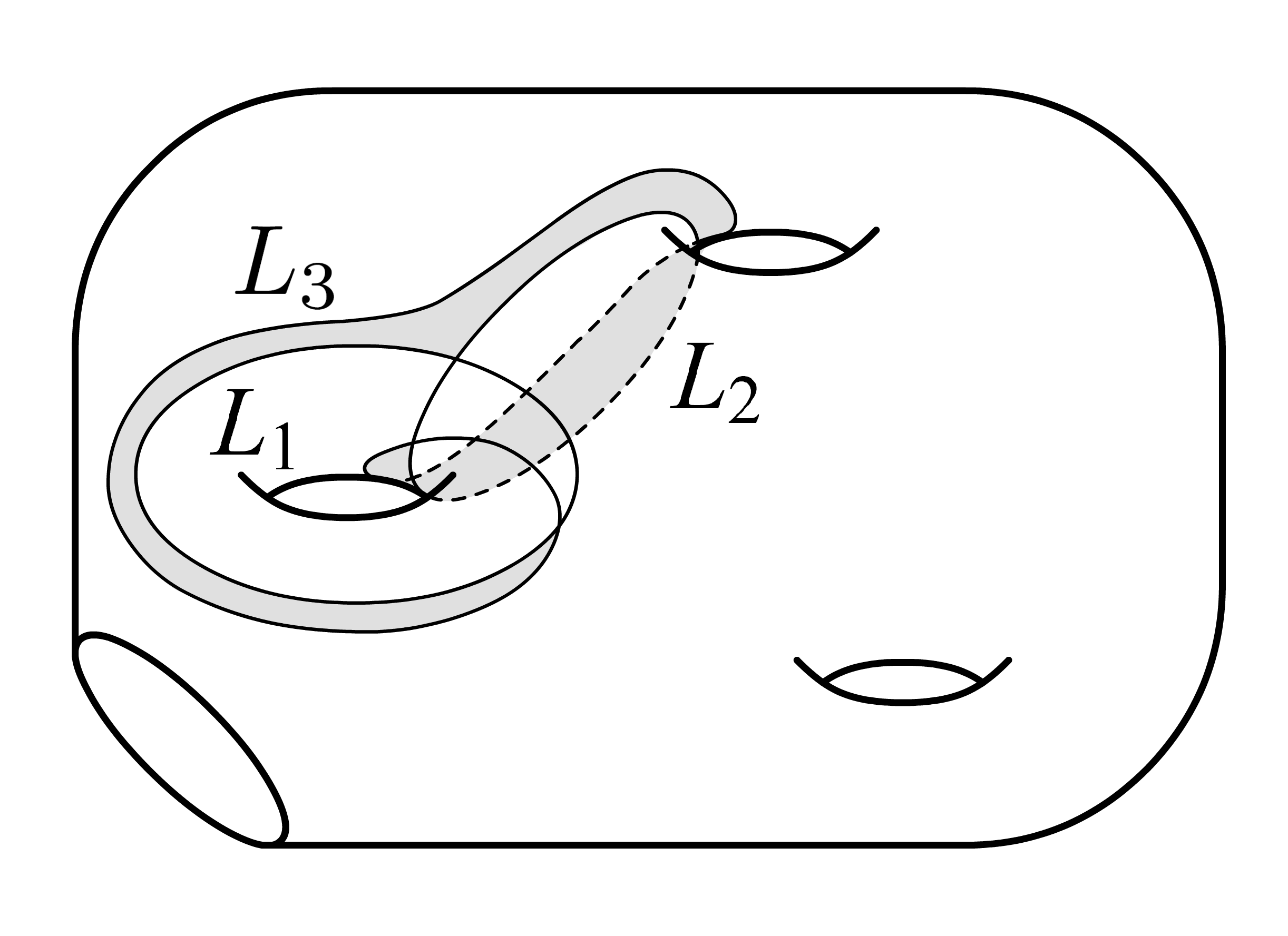}
\caption{Vanishing cycles of $\pi _1$}\label{fig:pi1}
\end{figure}

By Figure \ref{fig:pi1}, the dimension of the hom spaces $hom_{\mathcal{F}(\boldsymbol{L}_1)^{\to}} (L_1^{\#}, L_2^{\#})$,
$hom_{\mathcal{F}(\boldsymbol{L}_1)^{\to}} (L_1^{\#}, L_3^{\#})$, and 
$hom_{\mathcal{F}(\boldsymbol{L}_1)^{\to}} (L_2^{\#}, L_3^{\#})$ are all one. 
We can assume that $hom_{\mathcal{F}(\boldsymbol{L}_1)^{\to}}^0(L_1^{\#}, L_2^{\#})$ and $hom_{\mathcal{F}(\boldsymbol{L}_1)^{\to}}^0(L_2^{\#}, L_3^{\#})$ are non-trivial by shifting the gradings of Lagrangian branes.
Since we have  $\mathcal{M}^2 (p_{23}, p_{12}; p_{13}) \neq \varnothing$, we can conclude that $p_{13} \in hom_{\mathcal{F}(\boldsymbol{L}_1)^{\to}}^0(L_1^{\#},
L_3^{\#})$ by the same discussion in Remark \ref{rem:DegOfMu}.

Next, we compute the $\mu$'s. Since all the degrees of morphisms are zero, we have $\mu ^d = 0$ for $d \neq 2$. Thus the only nontrivial term we have to compute is $\mu ^2(p_{23}, p_{12})$.
In the case of $\pi _1$, our moduli space $\mathcal{M}^2(p_{23}, p_{12})$ has two elements, namely $u_1$ and $u_2$ (one is small and another is a big gray one in Figure \ref{fig:pi1}).
Let us consider the sign $(-1)^{s(u_i)}$.
All the morphisms have degree zero, hence the vertices do not contribute
to the sign. Since $\partial u_1 \cup \partial u_2 = L_1 \cup L_2 \cup L_3$ and $\partial u_1 \cap \partial u_2 = \{ p_{12}, p_{13}, p_{23} \}$, $\partial u_1$ and $\partial u_2$ shares three switching points $p_1, p_2, p_3$. 
Since three is odd, the signs of $u_1$ and $u_2$ are different. Thus we have $\mu ^2(p_{23}, p_{12}) = 0$.

Let us sum up the result in this subsubsection.
The Fukaya-Seidel category $\mathcal{F} (\pi_1)^{\to} \coloneqq \mathcal{F} (\boldsymbol{L}_1^{\#})^{\to}$  is isomorphic to the  $A_{\infty}$-category $\mathcal{A}_1$ defined as follows: ${\rm Ob}(\mathcal{A}_1) = \{ 1, 2, 3 \}$; the hom spaces are all zero but $hom_{\mathcal{A}_1}^0
(j, j) = k \cdot e_j$ for $j = 1, 2, 3$, and $hom_{\mathcal{A}_1}^0(j,
k) = k \cdot e_{jk}$ for $j < k$; and the higher composition maps are given by
$\mu _{\mathcal{A}_1}^d = 0$ for $d \neq 2$, $\mu _{\mathcal{A}_1}^2(e_k,
e_{jk}) = e_{jk}, \, \mu _{\mathcal{A}_1}^2(e_{jk}, e_j) = e_{jk}$, and 
$\mu _{\mathcal{A}_1}^2(e_{23}, e_{12}) = 0$.

\subsubsection{Computation of $\mathcal{F}(\boldsymbol{L}_2^{\#})^{\to}$}\label{subsubsec:CompOfFL2}

\begin{figure}[hbt]
\centering
\includegraphics[width=5cm]{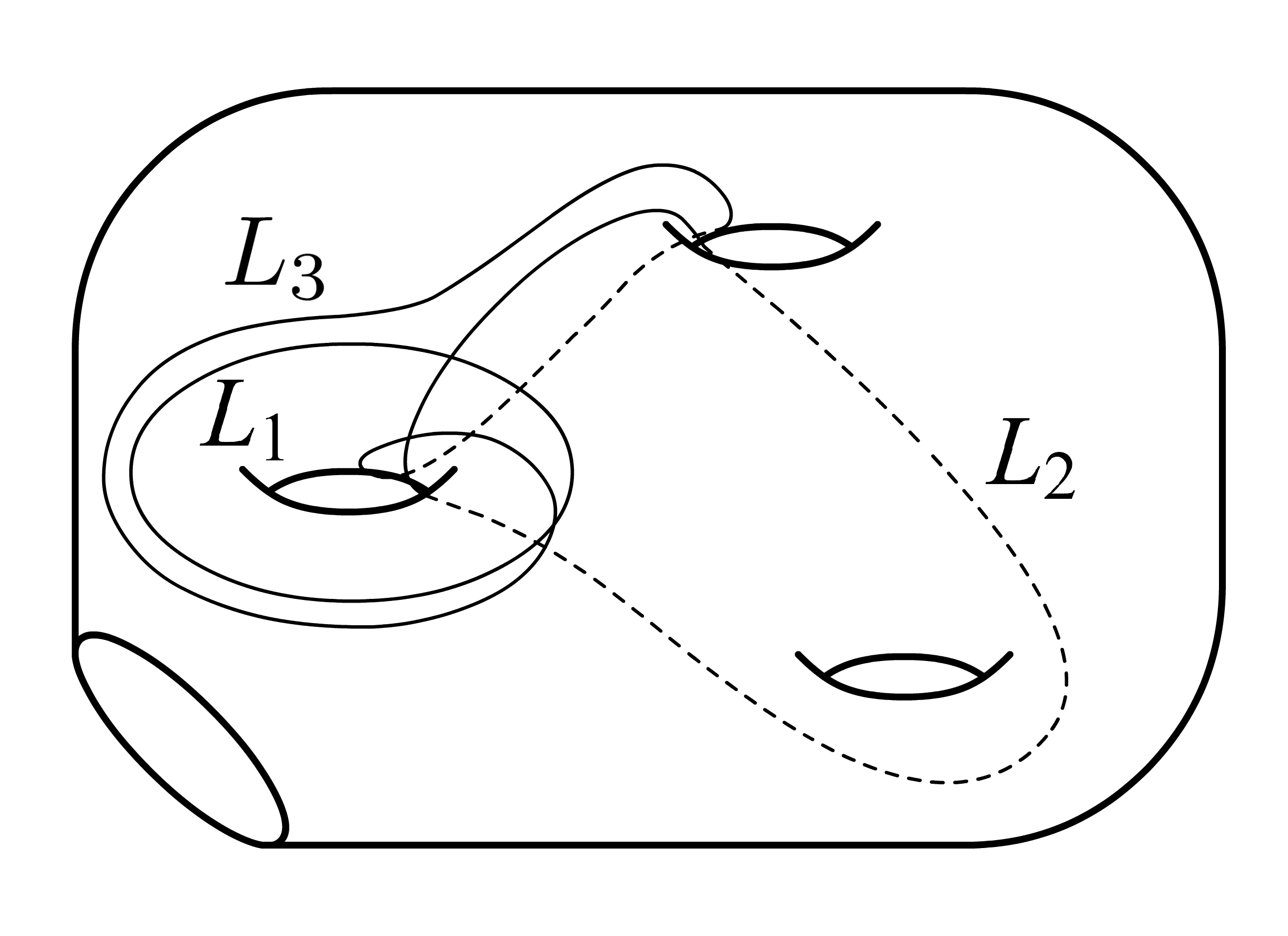}
\caption{Vanishing cycles of $\pi _2$}\label{fig:pi2}
\end{figure}

The situation of $\pi _2$ is all the same as $\pi _1$ but $\mathcal{M}^2(p_{23}, p_{12}; p_{13})$ has only single element $u$.
Hence, the only difference is the absence of cancellation in the computation of $\mu ^2$, so we have $\mu ^2(p_{23}, p_{12}) = \pm p_{13}$. If necessary, we change the switching point $p_i$ and have $\mu ^2(p_{23}, p_{12}) = +p_{13}$.

To sum up, we define the $A_{\infty}$-category $\mathcal{A}_2$ same as $\mathcal{A}_1$ but $\mu _{\mathcal{A}_2}^2(e_{23}, e_{12}) = e_{13}$, then we have $\mathcal{F}
(\pi_2)^{\to} \coloneqq \mathcal{F} (\boldsymbol{L}_2^{\#})^{\to} \cong \mathcal{A}_2$.

\subsubsection{Computation of $\mathcal{F}(\boldsymbol{L}_3^{\#})^{\to}$}\label{subsubsec:CompOfFL3}

\begin{figure}[hbt]
\centering
\includegraphics[width=5cm]{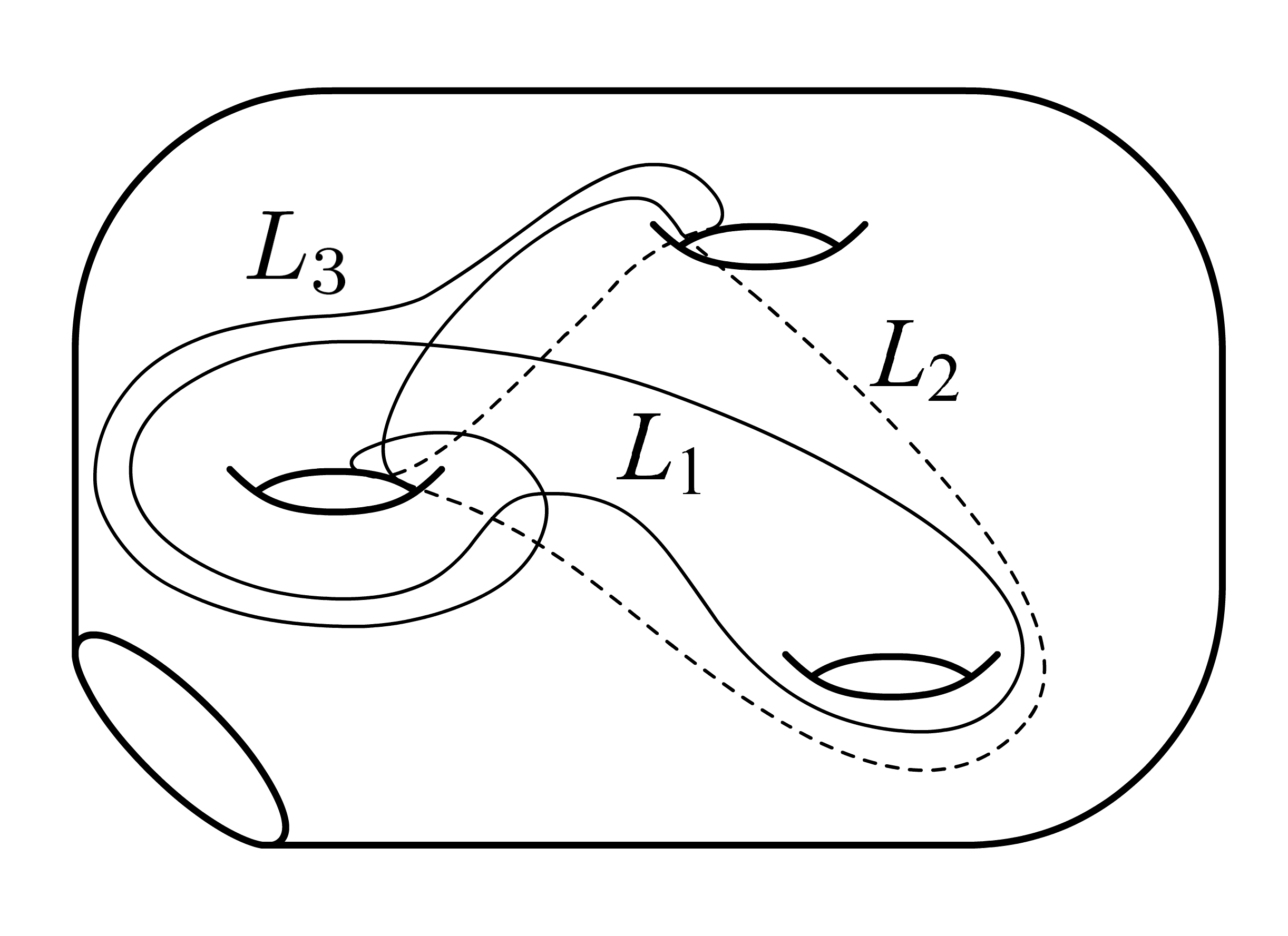}
\caption{Vanishing cycles of $\pi _3$}\label{fig:pi3}
\end{figure}

In this case, we have $\mathcal{M}^2(p_{23}, p_{12}) = \varnothing$, so $\mu$'s are all zero except for $\mu ^2$ with $e_i = 1_{L_i}$.
Because of the absence of elements in the moduli space, we can not conclude that $p_{13} \in hom_{\mathcal{F} (\boldsymbol{L}_3^{\#})^{\to}} ^0 (L_1^{\#}, L_3^{\#})$, and in fact this is not true. 
So, what we have to compute is the degree $|p_{13}|$. Let us fix a trivialization $X$ as follows.
To specify the trivialization up to homotopy, it is enough to fix the writhe of six $S^1$'s as in Figure \ref{fig:Sj}
since $\bigcup S_j$ is homotopy equivalent to $M \cong \Sigma _{3, 1}$.
Set $w_j \coloneqq w(S_j)$, then we have $w(L_1) = w_3 + w_5 -2$, $w(L_2) = w_2 + w_5 + 2$, and $w(L_3) = w_2 + w_3$ for some orientations of $L_i$. 

\begin{figure}[hbt]
\centering
\includegraphics[width=5cm]{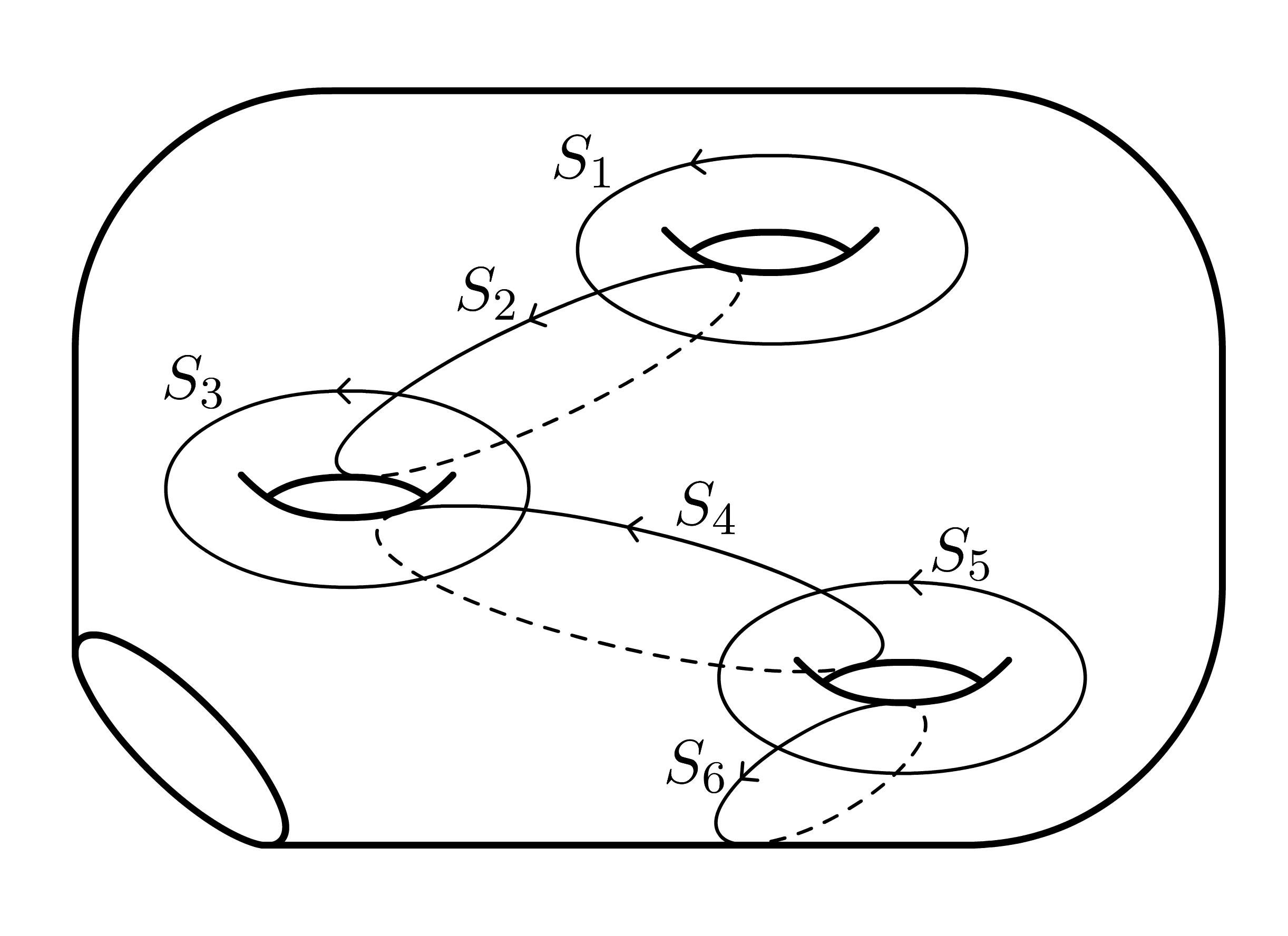}
\caption{$S_j$}\label{fig:Sj}
\end{figure}

The $\pm 2$ comes from the connection of two loops, for example, the process in Figure \ref{fig:concatenateS1} adds writhe by $-2$.

\begin{figure}[hbt]
\centering
\includegraphics[width=5cm]{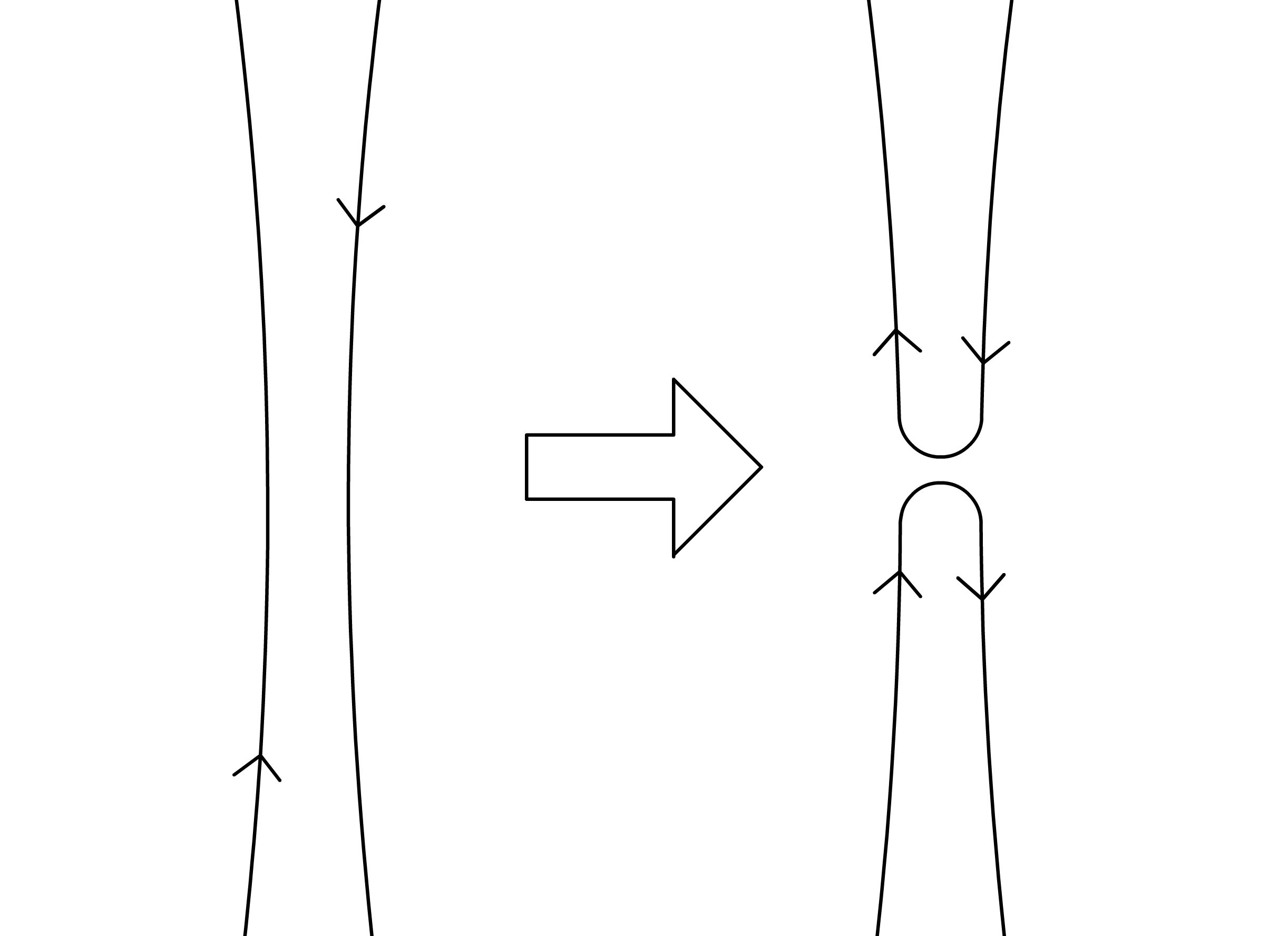}
\caption{Concatenating two $S^1$'s}\label{fig:concatenateS1}
\end{figure}

Now, since all the $L_i$'s are unobstructed,  the writhes must be zero, so we can conclude that $w_2 = -2, w_3 = 2$, and $w_5 = 0$.
Consider a (piecewise smooth) circle $C$ free homotopic to $S_5$ which starts from $p_{12}$, go along $L_2$ to $p_{23}$,  go along $L_3$ to $p_{13}$, and go back to $p_{12}$ along $L_1$.

\begin{lem}\label{lem:IndexFormula}

Let $M$ be an exact Riemann surface, $X$ be its trivialization, and $L^{\#}_0, L^{\#}_1, \dots , L^{\#}_n$ be unobstructed exact Lagrangian submanifolds.
We choose intersection points $y_0 \in L_0 \cap L_n$, $y_i \in L_{i-1} \cap L_i$ for $1 \leq i \leq n$. We set a piecewise smooth circle $C$ go along $L_0$ from $y_0$ to $y_1$, turn left and go along $L_1$ from $y_1$ to $y_2$, turn left ... and finally comes back to $y_0$.
Then, we have the following formula: $i(y_0) = i(y_1) + i(y_2) + \cdots + i(y_n) + (w(\widetilde{C}) - n)$. Here $i(y_0)$ is the degree of $y_0$ as a morphism from $L^{\#}_0$ to $L^{\#}_d$, $i(y_i)$ is the degree of $y_i$ as a morphism from $L^{\#}_{i-1}$ to $L^{\#}_i$, and $\widetilde{C}$ is smooth circle which is free homotopic to $C$.

\end{lem}

\begin{proof}

As in the case of remark \ref{rem:DegOfMu}, we only show in the case $i(y_1) = i(y_2) = \cdots = i(y_n) = 1$.
By the definition of writhe, we have $w(\widetilde{C}) - 1 < \alpha _d (y_0) - \alpha _0(y_0) < w(\widetilde{C})$.
So we can conclude $i(p_0) = w(\widetilde{C})$ and hence $i(y_0) = i(y_1) + i(y_2) + \cdots + i(y_n) + (w(\widetilde{C}) - n)$.

\hfill $\Box$

\end{proof}

By the above lemma, we have $i(p_{13}) = i(p_{12}) + i(p_{23}) + (w(S_5) - 2) = -2$.

Let us sum up the result in this subsubsection.
The Fukaya-Seidel category $\mathcal{F} (\pi_3)^{\to} \coloneqq \mathcal{F} (\boldsymbol{L}_3^{\#})^{\to}$ is isomorphic to  an $A_{\infty}$-category $\mathcal{A}_3$ defined as follows: ${\rm Ob}(\mathcal{A}_3)
= \{ 1, 2, 3 \}$; the hom spaces are all zero but $hom_{\mathcal{A}_3}^0
(j, j) = k e_j$ for $j = 1, 2, 3$, $hom_{\mathcal{A}_3}^0(j,
j+1) = ke_{j \, j+1}$ for $j = 1, 2$, $hom_{\mathcal{A}_3}^{-2}(1,
3) = kf_{13}$;
the higher composition maps are given by $\mu _{\mathcal{A}_1}^d = 0$ if $d \neq 2$, $\mu _{\mathcal{A}_1}^2(e_k,
x_{jk}) = x_{jk}, \, \mu _{\mathcal{A}_1}^2(x_{jk}, e_j) = x_{jk}$ where $x$ stands for $e$ or $f$, and 
$\mu _{\mathcal{A}_1}^2(e_{23}, e_{12}) = 0$.

\subsection{About the Hochschild cohomology}\label{subsec:Hochchild}

We use the following notation of Hochschild cohomology in \cite{Sh15} adapted to the notation of $A_{\infty}$-category of \cite{Se08}.

\begin{defn}[Hochschild cochain groups]\label{def:HochschildCochain}

Let $\mathcal{A}$ be an c-unital $A_{\infty}$-category.
We set $\mathcal{A}(X_s, X_{s-1}, \dots , X_0) \coloneqq hom_{\mathcal{A}}(X_{s-1}, X_s)[1] \otimes hom_{\mathcal{A}}(X_{s-2}, X_{s-1})[1] \otimes \cdots \otimes hom_{\mathcal{A}}(X_0 ,X_1)[1]$ for objects $X_0, \dots X_s \in Ob(\mathcal{A})$.
We define the Hochschild cochain groups of degree $r$ and length $s$  by 
\[ CC^r(\mathcal{A})^s \coloneqq \prod_{X_0, \dots , X_s \in Ob(\mathcal{A})} {\rm Hom}_{Gr(k)}^r(\mathcal{A}(X_s, X_{s-1}, \dots , X_0), \mathcal{A}(X_0, X_s))[-1], \]
and the Hochschild cochain groups $CC^r(\mathcal{A}) \coloneqq \prod_{s \geq 0} CC^r(\mathcal{A})^s$.
Here, $Gr(k)$ stands for the category of graded $k$ vector spaces.

\end{defn}

\begin{defn}[Gerstenhaber product, bracket]\label{def:GerstenhaberProd}

For $\psi \in CC^r(\mathcal{A})^s, \varphi \in CC^t(\mathcal{A})^u$, we define the Gerstenhaber product $\psi \star \varphi$ by
\[(\psi \star \varphi)(a_{s+u}, \dots , a_1) \coloneqq \sum _i (-1)^{\heartsuit_i}\psi (a_{s+u}, \dots , a_{i+u+1}, \varphi(a_{i+u}, \dots , a_{i+1}), a_i, \dots , a_1)\]
where $\displaystyle \heartsuit _i = (t-1)\sum_{1 \leq j \leq i} (|a_j|-1)$ and the Gerstenhaber bracket by $[\psi, \varphi] \coloneqq \psi \star \varphi - (-1)^{(r-1)(t-1)} \varphi \star \psi$.

\end{defn}

\begin{rem}

There are two remarkable facts: (i) the Gerstenhaber bracket is a graded Lie bracket, but the Gerstenhaber product is far from associative in general. (ii) The $A_{\infty}$-structure of $\mathcal{A}$ is nothing but an element $\mu ^{\bullet} \in CC^2(\mathcal{A})$ satisfying $\mu ^{\bullet} \star \mu ^{\bullet} = 0$, and  $\mu ^0 = 0$ (where $\bullet$ stands for length $s$).

\end{rem}

\begin{defn}[Hochschild cohomology]\label{def:HochschildCohom}

We define the Hochschild differential $M^1 \coloneqq CC^*(\mathcal{A}) \to CC^*(\mathcal{A})[1]$ by $M^1 \coloneqq [-, \mu ^{\bullet}]$ where $\mu ^{\bullet}$ is the $A_{\infty}$-structure. Then, $M^1$ defines a differential, i.e. $M^1 \circ M^1 = 0$. We define Hochschild cohomology groups by $HH^*(\mathcal{A}) \coloneqq H(CC^*(\mathcal{A}), M^1)$.

\end{defn}

\begin{rem}\label{rem:AInftyStrOfHochschildCochains}

The equation for Hochschild differential $M^1 \circ M^1 = 0$ follows from $\mu ^{\bullet} \star \mu ^{\bullet} = 0$.
In fact, we the Hochschild cochain complex $CC^*(\mathcal{A})$ and Gerstenhaber bracket form a differential graded Lie algebra (dgLa) structure, hence $HH^*(\mathcal{A})$ has a structure of graded Lie algebra.
The details can be found in \cite{Sh15}.

\end{rem}

By the definition, we have the following formula:

\begin{lem}\label{lem:CompOfM1AndCC}

Let $\mathcal{A}$ be an $A_{\infty}$-category such that $\mu$'s are all zero except for $\mu ^2$, and hom spaces are concentrated in even degrees.
The Hochschild differential can be described as follows:
for $f \in \prod_{X_0, X_1, \dots , X_d} CC^r(\mathcal{A})^d = {\rm Hom}_{Gr(k)}^r(\mathcal{A}(X_d, X_{d-1}, \dots , X_0), \mathcal{A}(X_d, X_0))[-1]$, 
\begin{align*}
M^1f(a_d, a_{d-1}, \dots , a_0) &= \mu^2(f(a_d, \dots , a_1), a_0) \\
&+ \sum_{1 \leq i \leq d} (-1)^i f(a_d, \dots , \mu^2(a_i, a_{i-1}), \dots a_0) \\
&+ (-1)^{d+1} \mu^2(a_d, f(a_{d-1}, \dots , a_0)).
\end{align*}
Moreover, if the hom spaces are concentrated in degree zero, then we have
$\displaystyle CC^d(\mathcal{A}) = \prod _{X_0, \dots X_d \in Ob(\mathcal{A})}{\rm Hom}_k(F\mathcal{A}(X_d, \dots X_0), F\mathcal{A}(X_d, X_0))$ for $d \geq 0$ and $CC^*(\mathcal{A}) = 0$ for $d < 0$,
where $F$ is the  forgetful functor from graded vector spaces to vector spaces and ${\rm Hom}_k$ stands for the space of linear maps.

\end{lem}

\subsection{Computation of the Hochschild cohomology}\label{subsec:CompOfHochschild}

In this subsection, we compute the Hochschild cohomology groups of $\mathcal{A}_i$ and finish the proof of theorem \ref{thm:ComputationalEx}.
Our composition map $\mu^2$ is naturally considered as a map $\displaystyle \mu^2 \colon \bigoplus _{X_0, X_1, X_2} \hom_{\mathcal{A}_i}(X_1, X_2) \otimes \hom_{\mathcal{A}_i}(X_0, X_1) \to \bigoplus _{X_0, X_2} \hom_{\mathcal{A}_i}(X_0, X_2)$, but in this subsection, we extend the domain of $\mu^2$ and consider it as
a map
\[ \mu^2 \colon \left( \bigoplus_{X_0, X_1} \hom_{\mathcal{A}_i}(X_0, X_1) \right)^{\otimes 2} \to \bigoplus_{X_0, X_1} \hom_{\mathcal{A}_i}(X_0, X_1).\]
Hence, we can write $\mu^2(b, a)$ for arbitrary $a \in \hom_{\mathcal{A}_i}(X_0, X_1)$ and $b \in \hom_{\mathcal{A}_i}(X_2, X_3)$ even if $X_1 \neq X_2$, and in this case, $\mu^2(b, a) = 0$.

\subsubsection{$HH^*(\mathcal{A}_2)$}\label{subsubsec:HHA2}

First, we compute the Hochschild cohomology groups of our SECOND $A_{\infty}$-category $\mathcal{A}_2$.
By lemma \ref{lem:CompOfM1AndCC}, we have $\displaystyle
CC^d(\mathcal{A}_2) = \prod _{X_0, \dots X_d \in Ob(\mathcal{A}_1)}{\rm Hom}_k(F\mathcal{A}_2(X_d,
\dots X_0), F\mathcal{A}_2(X_d, X_0))$ for $d \geq 0$ and $CC^{-1}(\mathcal{A}_2)
= 0$, where $F$ is the degree forgetting functor.
We will omit this $F$ for the sake of simplicity.

In the case of $\mathcal{A}_2$, ${\rm Hom}_k(\mathcal{A}_2(X_d, \dots X_0),
\mathcal{A}_2(X_d, X_0))$ is non-zero only when $X_0 \leq X_1 \leq \cdots
\leq X_d \in \{ 1, 2, 3\}$, and if so, this space is one-dimensional.
We write their canonical generator by $X_d X_{d-1} \cdots X_0$.

By Lemma \ref{lem:CompOfM1AndCC}, we have $CC^0(\mathcal{A}_2) = k \cdot 1 \oplus k \cdot 2 \oplus k \cdot 3$, where
$p = 1_p \in {\rm Hom}_k (k, \mathcal{A}_2(p, p))$ which is defined by $1_p(1_k) = e_p$ for the unit  $1_k$  in $k$. Also we have $\displaystyle CC^1(\mathcal{A}_2) = \bigoplus_{1\leq p \leq q \leq
3} k \cdot qp$, $\displaystyle CC^2(\mathcal{A}_2) = \bigoplus_{1\leq p \leq
q \leq r \leq 3} k \cdot rqp$, and so on.
The possibly non-zero part of $M^1(1)$ is:
\begin{align*}
M^1(1)(e_1) &= \mu^2(1_1(1_k), e_1) - \mu^2(e_1, 1_1(1_k)) \\
&= e_1 - e_1 = 0, \\
M^1(1)(e_{1q}) &= \mu^2(1_1(1_k), e_{1q}) - \mu^2(e_{1q}, 1_1(1_k)) \\
&= 0 - e_{1q} = -e_{1q} \, \, \, (q = 2, 3).
\end{align*}
Here the $0$ in the fourth line comes from the discordance of the target of $e_{1q}$ and source of $e_1$.
Hence we have $M^1(1) = - 21 - 31$.

In the same way, the possibly non-zero part of $M^1(2)$ is:
\begin{align*}
M^1(2)(e_2) &= 0, \\
M^1(2)(e_{23}) &= -e_{23}, \\
M^1(2)(e_{12}) &= e_{12}.
\end{align*}
Hence we have $M^1(2) = 21 - 32$. By the same computation, we have $M^1(3)
=  31 + 32$.
Finally, we can conclude that $HH^0(\mathcal{A}_1) = \ker M^1 = k \, (1 + 2 + 3) \cong
k$.

Next, we compute $M^1 \colon CC^1 \to CC^2$.
To compute this differential, we compute its dual $m \colon (CC^2)^{\vee} \to (CC^1)^{\vee}$.
Now, we have $\mathcal{A}_2(q, p) \cong k$ for $p \leq q$ and there are canonical generators $e_{pq}$, and $CC^d(\mathcal{A}_2)$ are finite dimensional, hence there is canonical isomorphism $CC^d(\mathcal{A}_2) \cong \prod \mathcal{A}_2(X_d, X_{d-1}, \dots , X_0)^{\vee}$.
In this sence,  we identify $m$ with a map $m \colon \prod _{X_0, X_1, X_2} \mathcal{A}_2(X_2, X_1, X_0) \to \prod _{X_0, X_2} \mathcal{A}_2(X_2, X_0)$.
Let us write the canonical generator of $\mathcal{A}_2(q, p)$ again by $qp$, and $\mathcal{A}_2(r, p, q)$ by $rqp$ for $p \leq q \leq r$.
Thus we have the following formula:

\begin{lem}

$m(rqp) = rq - rp + qp$.

\end{lem}

We won't prove this lemma but see one example of $m(321)$.
For $f = 11, 21, 31, 22, 32, 33 \in CC^2(\mathcal{A}), a_1 \in \mathcal{A}_2(2, 1)$, and $a_2 \in \mathcal{A}_2(3, 2)$, we have $M^1f(a_2, a_1) = f(a_2) \circ a_1 - f(a_2 \circ a_1) + a_2 \circ f(a_1)$.
If this map $(a_2, a_1) \mapsto f(a_2) \circ
a_1$ is non-zero, then $f$ must be equals to $32$. In the same manner, when we consider the case of second  term of $M^1f$, we have $f = 31$, and from the third term, we have $f = 21$.
By this computation, we finally have $m(321) = 32 - 31 + 21$.

By the above lemma, we have the formula of $M^1$:
\begin{align*}
M^1(11) &= 111 + 211 + 311, \\
M^1(21) &= 321, \\
M^1(31) &= -321, \\
M^1(22) &= 221 + 222 + 322, \\
M^1(32) &= 321, \\
M^1(33) &= 331 + 332 + 333.
\end{align*}
Hence, we have $\ker M^1 = \{ a_{21} \, 21 + a_{31} \, 31 + a_{32} \, 32 \, | \, a_{21} - a_{31} + a_{32} = 0 \} = {\rm im}M^1$ in $CC^1(\mathcal{A}_2)$ hence we can conclude that $HH^1(\mathcal{A}_2) = 0$.

\subsubsection{$HH^*(\mathcal{A}_1)$}\label{subsubsec:HHA1}

We can compute that $HH^0(\mathcal{A}_1) \cong k$, since $\mathcal{A}_1$ and $\mathcal{A}_2$ have the same hom spaces and the same value of $\mu ^2$ when we substitute $e_i$ in one of the two entries.
Now we compute the $HH^1(\mathcal{A}_1)$.
First, we compute $m$ as in the case of $\mathcal{A}_2$.
Because $\mu ^2(e_{23}, e_{21}) = 0$, we have $m(321) = 0$ by the following computation: for $a_1 \in \mathcal{A}_2(2, 1)$ and $a_2 \in \mathcal{A}_2(3, 2)$, we have $M^1f(a_2, a_1) = \mu^2(f(a_2),
a_1) - f(\mu^2(a_2, a_1)) + \mu^2(a_2, f(a_1))$, but each term contains $\mu ^2(e_{23}, e_{21})$ hence vanishes.
We can check that this is the only difference between $m$ of $\mathcal{A}_1$ and $\mathcal{A}_2$. Finally, we have the following formula:
\begin{align*}
M^1(11) &= 111 + 211 + 311, \\
M^1(21) &= 0, \\
M^1(31) &= 0, \\
M^1(22) &= 221 + 222 + 322, \\
M^1(32) &= 0,  \\
M^1(33) &= 331+332+333.
\end{align*}
In $CC^1(\mathcal{A}_1)$, we have $\ker M^1 = k \, 21 \oplus k \, 31 \oplus k \, 32$ and the image of $M^1 \colon CC^0 \to CC^1$ is the same as that in the case of $\mathcal{A}_2$, hence we get $HH^1(\mathcal{A}_1) \cong k$.

\subsubsection{$HH^*(\mathcal{A}_3)$}\label{subsubsec:HHA3}

In this case, the situation is  different since the degree of the morphism from $1$ to $3$ is $-2$. Because of that, the degree of the element $3 \cdots 32 \cdots 21 \cdots 1 \in CC^*(\mathcal{A}_3)$ is smaller than that of $3 \cdots 32 \cdots 21 \cdots 1 \in CC^*(\mathcal{A}_1)$ by two if the character string contains each character at least once.
This is all of the differences of degrees.
So we have
$CC^d(\mathcal{A}_3) = 0$ for $d < 0$, 
$CC^0(\mathcal{A}_3) = k \, 1 \oplus k \, 2 \oplus k \, 3 \oplus k \, 321$, and $\displaystyle CC^1(\mathcal{A}_3) = \bigoplus _{1 \leq p \leq q \leq 3} k \, qp \oplus k \, 3211 \oplus k \, 3221 \oplus k \, 3321$.
We can compute $m \colon CC^1(\mathcal{A}_3)^{\vee} \to CC^0(\mathcal{A}_3)^{\vee}$ as follows:
\begin{align*}
m(qp) &= -p + q, \\
m(3211) &= 321 - 321 + 0 - 0 = 0, \\
m(3221) &= 0 - 321 + 321 - 0 = 0, \\
m(3321) &= 0 - 0 + 321 - 321 = 0,\\ 
\end{align*}
as in the previous case.
The zeros above are due to $\mu ^2(e_{23}, e_{12}) = 0$.
Hence we have, 
\begin{align*}
M^1(1) &= -21 - 31, \\
M^1(2) &= 21 - 32, \\
M^1(3) &= 31 + 32, \\
M^1(321) &= 0. 
\end{align*}
Thus we have $HH^0(\mathcal{A}_3) = \ker M^1 = k \, (1 + 2 + 3) \oplus k \, 321 \cong k^2$.

By the same computation, we have the values of $m$ for the basis of $CC^2(\mathcal{A}_3)^{\vee}$ as follows:
\begin{align*}
m(rqp) &= rq - rp + qp \, \, \, ( \text{for } (p, q, r) \neq (1, 2, 3) \, ), \\
m(32111) &= 3211 - 3211 + 3211 + 0 - 0 = 3211, \\
m(32211) &= 3221 - 3221 + 3211 - 3211 + 0 = 0, \\
m(32221) &= 0 - 3221 + 3221 - 3221 + 0 = -3221, \\
m(33211) &= 3321 - 3321 + 0 - 3211 + 3211 = 0, \\ 
m(33221) &= 0 - 3321 + 3321 - 3221 + 3221 = 0, \\
m(33321) &= 0 - 0 + 3321 - 3321 + 3321 = 3321,
\end{align*}
again the above zeros are due to $\mu ^2(e_{23}, e_{12}) = 0$. Thus we have
\begin{align*}
M^1(11) &= 111 + 211 + 321 \\
M^1(21) &= 0 \\
M^1(31) &= 0 \\
M^1(22) &= 221 + 222 + 322 \\
M^1(32) &= 0 \\
M^1(33) &= 331 + 332 + 333 \\
M^1(3211) &= 32111 \\
M^1(3221) &= -32221 \\
M^1(3321) &= 33321.
\end{align*}
Hence, in $CC^1(\mathcal{A}_2)$ we obtain $\ker M^1 = k \, 21 \oplus k \, 31 \oplus k \, 32$ and ${\rm im} \,  M^1 = k \, (21 +\ 31) \oplus k(-21 + 32)$, finally we have $HH^1(\mathcal{A}_3) \cong k$.
This completes all of the computation.

The Milnor lattices, which cannot distinguish these three Lefschetz fibrations, involves only the information of the intersections of the two vanishing cycles.
On the other hand, in the case of the Fukaya-Seidel categories, the positional relation of three or more vanishing cycles is taken into account. We can think that this is the reason why the Fukaya-Seidel categories can distinguish the above three Lefschetz fibrations.
Thus, we can say that the Fukaya-Seidel categories capture somewhat higher relations of vanishing cycles, which are, by definition,  polygons enclosed by the vanishing cycles, as the higher composition maps $\mu$'s.



\end{document}